\documentclass[11pt,final]{amsart}
\usepackage{geometry,enumitem}
\usepackage{accents}
\usepackage{graphicx}
\usepackage{color}
\usepackage{subcaption}

\usepackage{cite}
\usepackage{hyperref}
\newtheorem{theorem}{Theorem}[section]
\newtheorem{proposition}[theorem]{Proposition}
\newtheorem{lemma}[theorem]{Lemma}
\newtheorem{corollary}[theorem]{Corollary}
\newtheorem{definition}{Definition}[section]

\numberwithin{equation}{section}

\numberwithin{equation}{section}
\usepackage{mathrsfs}

\DeclareMathOperator{\spt}{spt}

\begin{document}
\title[Optimal controlled transports subject to import/export tariffs]
{Optimal controlled transports with free end times subject to import/export tariffs}
\author[S. Dweik, N. Ghoussoub and  A. Z. Palmer]{Samer Dweik, Nassif Ghoussoub and Aaron Zeff Palmer}
\address{Department of Mathematics, University of British Columbia, Vancouver BC Canada V6T 1Z2.}
\email{dweik@math.ubc.ca, nassif@math.ubc.ca, azp@math.ubc.ca}

\begin{abstract} 
We analyze controlled mass transportation plans with free end-time that minimize the transport cost induced by the generating function of a Lagrangian within a bounded domain, in addition to costs incurred 
as export and import tariffs at entry and exit points on the boundary.  
We exhibit a dual variational principle \`a la Kantorovich, that takes into consideration the additional tariffs. We then show that the primal optimal transport problem has an equivalent Eulerian formulation whose dual involves the resolution of a
 Hamilton-Jacobi-Bellman quasi-variational inequality with non-homogeneous boundary conditions. This allows us to prove existence and to describe the solutions
 for both the primal optimization problem and its Eulerian counterpart.  
\end{abstract}
\maketitle

\tableofcontents

 \section{Introduction} \label{1}
 
 Let $\mu^+$ be a positive measure on a bounded domain $\Omega$ of ${\mathbb R}^d$ that encodes both the location and the supply of goods produced by certain factories, and let $\mu^-$ be another positive measure on $\Omega$ that represents the location  of some customers as well as their consumption requirement for these goods. Assuming a function $c(x,y)$ describes the transport cost of a mass unit $x$ to $y$, and subject to the mass balance condition $\mu^+( \Omega)= \mu^-(\Omega),$ standard Monge-Kantorovich theory formulates the least costly transport plan as a solution for the following optimization problem,
 \begin{equation} \label{Kantorovich classique}
{\mathcal T}(\mu^+, \mu^-):= \inf\bigg\{\int_{\Omega \times \Omega} c(x,y)\,\mathrm{d}\pi(x,y) \,:\,\pi \in \mathcal{P}(\mu^+,\mu^-)\bigg\},
 \end{equation}
 where 
 $$\mathcal{P}(\mu^+,\mu^-):=\bigg\{\pi \in \mathcal{M}^+(\Omega \times \Omega)\,:\, \pi_x=\mu^+ \,\,\;\mbox{and}\;\,\,\pi_y=\mu^-\bigg\},$$
 and $\pi_x$ and $\pi_y$ are the two marginals of $\pi$ on $\Omega$. Note that if $c$ is continuous, Problem \eqref{Kantorovich classique} is the relaxed Kantorovich version of the so-called Monge Problem \cite{Monge}, 
  \begin{equation*}
{\mathcal T}(\mu^+, \mu^-)= \inf\bigg\{\int_{\Omega} c(x,T(x))\,\mathrm{d}\mu^+(x) \,:\,T_{\#}\mu^+=\mu^-\bigg\}, 
 \end{equation*}
 where $T_{\#}\mu^+=\mu^-$ means that the transformation $T$ pushes $\mu^+$ onto $\mu^-$, i.e., $\mu^-(A)=\mu^+(T^{-1}(A))$ for every Borel subset $A$. Under mild conditions on the transport cost $c$, Problem \eqref{Kantorovich classique} has the following dual formulation
 (see, for instance, \cite{8,11}):
\begin{equation} \label{dual classique}
\mathcal{D}(\mu^+,\mu^-):=\sup\bigg\{\int_\Omega \phi^-\,\mathrm{d}\mu^- - \int_\Omega \phi^+\,\mathrm{d}\mu^+\,\,:\,\,\phi^\pm \in C(\Omega),\,(- \phi^+) \oplus \phi^- \leq c\bigg\}.
\end{equation}
We refer to \cite{Ambrosio,Brenier1,Brenier2,  
8,11} for an introduction to optimal transport theory, its history, and its main results. 
 
 Our goal in this paper is to study a variant of this problem already considered by several authors \cite{Malusa,DweSan,Dweik,Dweik7,7}, namely the case where one considers transport plans that move all the products in $\mu^+$ and cover all the needs of the consumers in $\mu^-$, with the possibility of importing and exporting products across the boundary of $\Omega$, provided export (respect import) tariffs are paid  in addition to the transport cost: one is then charged an extra cost $-\,\psi^-(y)$ for each unit that comes out from a point $y \in \partial\Omega$ (the export tax) and a tariff $+ \,\psi^+(x)$ for each unit that enters at the point $x \in \partial\Omega$ (the import tax).  Note that the usual balance condition on $\mu^+, \mu^-$ is not imposed here since we can import and export through the boundary at will, if necessary. This means that $\partial\Omega$ can be considered an infinite reserve/repository from which one can import as much product as need be, and to which one can export as much mass as necessary, provided that one pays the import/export taxes in addition to the transportation cost. To formulate the problem, one considers 
 the set
$$\mathcal{P}_0(\mu^+,\mu^-):=\bigg\{\pi \in \mathcal{M}^+(\Omega\times\Omega)\,:\,(\pi_x)_{\,|\accentset{\circ}{\Omega}}=\mu^+,\,(\pi_y)_{\,|\accentset{\circ}{\Omega}} =\mu^-\bigg\},$$
and minimizes the quantity
{\small 
 \begin{equation} \label{Kantorovich with boundary costs}
{\mathcal T}_1(\mu^\pm; \psi^\pm):=\inf\bigg\{\int_{\Omega \times \Omega}c\,\mathrm{d}\pi
+\int_{\partial\Omega}\psi^+\,\mathrm{d}\pi_x
  -\int_{\partial\Omega}\psi^-\,\mathrm{d}\pi_y 
   \,:\,\pi \in \mathcal{P}_0(\mu^+,\mu^-) \bigg\}.
   \end{equation}
   }
   \hspace{-0.5em} It is important to assume the following ``no arbitrage condition," which makes sure that there is no advantage to transporting goods from the boundary to other boundary locations:
	\begin{equation}\label{g_1g_2}
		\psi^-(y)-\psi^+(x)\leq c(x,y) \ \,\,\,\hbox{for all \,\,$(x,y)\in \partial \Omega \times \partial \Omega$}.
	\end{equation}
Just like in classical Monge-Kantorovich theory, Problem \eqref{Kantorovich with boundary costs} has then a dual formulation where the Kantorovich potentials now satisfy certain   
boundary conditions. More precisely, one can show that 
\begin{equation}\label{weakduality1}
{\mathcal T}_1(\mu^\pm; \psi^\pm)={\mathcal D}_1(\mu^\pm; \psi^\pm),
\end{equation}
where
{\small 
\begin{equation}   \label{dual with boundary costs}
{\mathcal D}_1(\mu^\pm; \psi^\pm):=   \sup_{\varphi^\pm \in C(\Omega)}\left\{\int_\Omega \varphi^- \mathrm{d} \mu^- - \int_\Omega \varphi^+ \mathrm{d}\mu^+
: \begin{array}{l}\psi^- \leq \varphi^-\leq \varphi^+ \leq \psi^+\; \,\mbox{on}\;\, \partial\Omega,\\
(-\varphi^+) \oplus \varphi^- \leq c\end{array}\right\}.
\end{equation}
}
In this paper, we shall conisder transport costs $c(x,y)$ given by the minimal value of some optimal control problem between $x$ and $y$. Classical Monge-Kantorovich problems associated with such costs were considered by Agrachev and Lee \cite{A-L} for trajectories with fixed end-times, and by Ghoussoub-Kim-Palmer \cite{Aaron} for the case where end-times also need to be determined.  It is the latter set-up that we shall consider here, that is when 
 $$c(x,y)=\inf_{\tau,\,u}\bigg\{\int_0^\tau L(t,\gamma(t),u(t))\,\mathrm{d}t\,:\,\dot\gamma(t)=k(t,\gamma(t),u(t)),\,\gamma(0)=x,\,\gamma(\tau)=y\bigg\},$$
 where $L$ is a Lagrangian, $u$ is a set of controls and $k$ is a functional that determines the dynamics. This leads us to formulate an Eulerian version of the problem that dynamically describes the movement of goods. It also calls for finding optimal stopping times for their delivery.

To give an Eulerian formulation for the primal problem \eqref{Kantorovich with boundary costs}, we follow ideas in \cite{Aaron} and consider the set 
 $$\mathcal{E}_0(\mu^+,\mu^-):=\left\{(\rho,\eta)\,:\,\begin{array}{l} \rho:\mathbb{R}^+ \mapsto \mathcal{M}^+(\mathbb{R}^d\times U),\,\eta \in \mathcal{M}^+(\mathbb{R}^+ \times \mathbb{R}^d), \\(\rho,\eta) \,\,\mbox{solves}\,\,\eqref{stopping definitions}\end{array}\right\},$$
where the equations for $(\rho,\eta)$ are expressed formally as
\begin{equation}\label{stopping definitions}
\begin{cases}
\eta + \partial_t\bigg(\int_U \mathrm{d}\rho\bigg) + \nabla \cdot \bigg(\int_U k\,\mathrm{d}\rho\bigg)=0,\\
\rho_{0}(\cdot,U)=\mu^+\,\,\, \;\mbox{on}\,\, \; \mathbb{R}^d\backslash \partial\Omega,\\
\eta(\mathbb{R}^+,\cdot)=\mu^-\,\,\, \; \mbox{on}\,\, \; \mathbb{R}^d\backslash \partial\Omega.
\end{cases}
\end{equation}
This means that the boundary parts of $\rho_{0}(\cdot,U)$ and $\eta(\mathbb{R}^+,\cdot)$ are now unknown (in fact, these two positive measures represent the import/export masses on the boundary $\partial\Omega$). The Eulerian formulation of Problem \eqref{Kantorovich with boundary costs} then becomes 
\begin{equation}\label{Eulerian classique with boundary}
{\mathcal T}_2(\mu^\pm; \psi^\pm):=\inf\limits_{(\rho,\eta) \in \mathcal{E}_0(\mu^+,\mu^-)}\left\{\begin{array}{l}\displaystyle\int_{\mathbb{R}^+}\displaystyle\int_{\mathbb{R}^d \times U} L(t,x,u)\,\mathrm{d}\rho_t(x,u)\,\mathrm{d}t\\
 \qquad\,\,\,\,+\,\,\, \displaystyle\int_{\partial\Omega\times U} \psi^+(x) \,\mathrm{d}\rho_0(x,u)\\
 \qquad\qquad\,\,\,\,\,\, -\,\,\, \displaystyle\int_{\mathbb{R}^+\times \partial\Omega} \psi^-(y)\,\mathrm{d}\eta(t,y)\end{array}\right\}.
\end{equation}
Similarly to \cite{Aaron}, we then consider the dual of the Eulerian formulation \eqref{Eulerian classique with boundary},  which  involves the resolution of the following Hamilton-Jacobi-Bellman quasi-variational inequality but, now, with certain non-homogeneous boundary conditions:
\begin{equation}\label{HJB with boundary}
\begin{cases}
\partial_t J^+(t,x) + k(t,x,u) \cdot \nabla J^+(t,x) \leq L(t,x,u)\,\,\,&\ (t,x,u) \in \mathbb{R}^+ \times \mathbb{R}^d \times U,\\
\varphi^- (x) \leq J^+(t,x) \,\,\, &\ (t,x) \in \mathbb{R}^+ \times \mathbb{R}^d,\\
\psi^-(x) \leq \varphi^- (x) \,\,\, &\ x \in \partial\Omega,\\
J^+(0,x) \leq \psi^+(x) \,\,\, &\ x \in \partial\Omega.
\end{cases}
\end{equation}\\
  The second dual problem is now, 
{\small
\begin{equation} \label{dual Eulerian with boundary}
{\mathcal D}_2(\mu^\pm; \psi^\pm):=\sup\bigg\{\int_\Omega \varphi^-\mathrm{d}\mu^- - \int_\Omega J^+(0,\cdot)\,\mathrm{d}\mu^+: 
\begin{array}{l}\varphi^- \in C(\mathbb{R}^d),\,J^+ \in C^1(\mathbb{R}^+ \times \mathbb{R}^d),\\
(J^+,\varphi^-)\;\,\mbox{solves}\,\;\eqref{HJB with boundary} \end{array}\bigg\}.
\end{equation}
}\\
The first goal of this paper is to prove that --under natural conditions-- the following equalities hold, 
\begin{equation}
{\mathcal T}_1(\mu^\pm; \psi^\pm)={\mathcal T}_2(\mu^\pm; \psi^\pm)={\mathcal D}_1(\mu^\pm; \psi^\pm)={\mathcal D}_2(\mu^\pm; \psi^\pm).
\end{equation}
We will then show, under additional hypotheses, that minimizers of ${\mathcal T}_1(\mu^\pm; \psi^\pm)$ and ${\mathcal T}_2(\mu^\pm; \psi^\pm)$ are given by transport maps, determined by a Hamiltonian flow terminating along the free boundary of the optimal dual potentials.

We note that the above model is a particular case of a more general setting, where we can consider the reserve mass is taken from a prescribed set $K^+$ with cost $\psi^+$, and can be deposited in the set $K^-$ with cost $-\psi^-$, where $K^+$ and $K^-$ are two compact sets of $\mathbb{R}^d$. The admissible set of transport plans is then
	\begin{equation}
		\mathcal{P}_K(\mu^+,\mu^-)=\bigg\{\pi \in \mathcal{M}^+(\mathbb{R}^d\times \mathbb{R}^d): (\pi_x)_{|\mathbb{R}^d\backslash K^+}=\mu^+,\ (\pi_y)_{|\mathbb{R}^d\backslash K^-}=\mu^-\bigg\}, 
	\end{equation}
	and the new variational problem becomes 
	\begin{equation}
		\min\bigg\{\int_{\mathbb{R}^d\times \mathbb{R}^d} c(x,y)d\pi+\int_{K^+}\psi^+d\pi_x-\int_{K^-}\psi^- d\pi_y : \pi\in \mathcal{P}_K(\mu^+,\mu^-)\bigg\}.
	\end{equation}
	Again, we assume that the costs $\psi^+$ and $\psi^-$ satisfy the no arbitrage assumption (\ref{g_1g_2}), which becomes
	\begin{equation}
		\psi^-(y)-\psi^+(x)\leq c(x,y),\ \hbox{for all $(x,y)\in K^+ \times K^-$}.
	\end{equation}
	The same analysis as above can be carried out and a sketch is given in Section \ref{Sec.6}. Note that the above model is the particular case where 
	$K^+=K^-=\partial \Omega$.\\

  Finally, in Section \ref{Sec.7}, a one-dimensional example is presented to demonstrate how the structure of the problem can be utilized in a solution.\\

This paper is organized as follows. In Section \ref{Sec. 2}, we introduce the control problem needed to define the transportation cost and the conditions under which the existence of optimal trajectories and the continuity of the transport cost are guaranteed. In Section \ref{Sec. 3}, we analyze in details the primal transportation problem with  tariff costs, prove existence of an optimal transport plan and give a proof for the first duality principle. In Section \ref{Sec. 4}, we introduce an equivalent Eulerian formulation and its dual, establish the existence of an optimal admissible pair and show the equivalence with the primal problem. In Section \ref{Sec.5} we identify the optimal stopping times, while in Section \ref{Sec.6} we sketch a proof of the more general setting where the tariff costs $\psi^+$ (resp., $-\psi^-$) are incurred when goods are taken from (resp., deposited in) prescribed locations $K^+$ (resp., $K^-$). Section \ref{Sec.7} presents a simple one-dimensional example.

\section{Free end-time optimal control problem: preliminaries} \label{Sec. 2}
In this section, we consider the optimal control problem that we will use to define the transport cost $c(x,y)$, between two points $x$ and $y$.  
We assume that the trajectory from a point $x$ to another one $y$ is submitted to a non-autonomous control system, the time-dependence of the dynamic comes, for instance, from the interaction between particles. This control problem is said to be with \,{\it{free end time}}\, since the terminal time of the trajectories from $x$ to $y$ is not fixed, but is the first time at which they reach the point $y$. 
So, we consider control systems whose state equation is of the form
\begin{equation} \label{control system}
\left\{
\begin{aligned}
\dot{\gamma}(t) & = k(t,\gamma(t),u(t)), \,\,\,\text{for a.e.\ }t \geq t_0, \\
\gamma(t_0) & = x,
\end{aligned}
\right.
\end{equation}
where $\gamma(t) \in \mathbb{R}^d$ is the state, the continuous function $k: \mathbb{R}^+ \times \mathbb{R}^d \times \mathbb{R}^d \to \mathbb{R}^d$ is called the \emph{dynamic} of the system, $t_0 \in \mathbb{R}^+,\,x \in \mathbb{R}^d$, and $u: \left[t_0,+\infty \right[ \mapsto U$ is a measurable function (which is called a \emph{control} and $U$ is the {\it{control set}}).
We list some basic assumptions on the dynamic $k$ and the control set $U$:\smallskip

(H0) \quad The control set $U$ is compact,
 
(H1)  \quad  $k$ is bounded, i.e., $|k(t,x,u)| \leq \kappa\,\,\mbox{for all}\,\,(t,x,u) \in \mathbb{R}^+ \times \mathbb{R}^d \times U$,
 
(H2)  \quad  $k$ is continuous with respect to $u$ and is $C$-Lipschitz in the other variables,\\ 
\hspace*{4em} i.e., $\exists \,C > 0 \,\text{ such that }|k(t,x,u) - k(s,x^\prime,u)| \leq C(|t-s| +| x - x^\prime|)\\
\hspace*{4em}\,\,\text{for all }\,t,\,s \in \mathbb{R}^+,\,x,\,x^\prime \in \mathbb{R}^d\,\text{ and } u \in U$.\\

Note that Assumption (H2) ensures the existence of a unique global solution to the state equation \eqref{control system} for any choice of $t_0,\,x$ and $u$. We shall denote such a solution of \eqref{control system} by $\gamma^{t_0,x}_u$ and we call it an {\it{admissible trajectory}} of the system, corresponding to the initial condition $\gamma(t_0)=x$ and to the control $u$.
For a given trajectory $\gamma=\gamma^{t_0,x}_u$ of \eqref{control system}, we set
$$\tau^{t_0,x,y}_u=\inf\{\tau \geq 0 : \gamma^{t_0,x}_u(t_0 +\tau)=y\},\,\,\mbox{for every}\,\,y \in \mathbb{R}^d,$$
with the convention that $\tau^{t_0,x,y}_u=+\infty$\, if \,$\gamma^{t_0,x}_u(t_0 +\tau) \neq y$, for all \,$\tau \geq 0$. 

Now, our optimal control problem consists of choosing the control strategy $u$ in the state equation \eqref{control system} in order to minimize a given cost given by a Lagrangian. For that, consider  $L: \mathbb{R}^+ \times \mathbb{R}^d \times \mathbb{R}^d \mapsto \mathbb{R}^+$ to be a given continuous function. For every $x,\,y \in \mathbb{R}^d$, we minimize the cost
\begin{equation} \label{controlquantity}
J^{t_0,x,y}(u)=\int_{t_0}^{\,t_0 + \tau_u^{t_0,x,y}}L(t,\gamma_u^{t_0,x}(t),u(t))\,\mathrm{d}t,
\end{equation}
among all controls $u$. A control $u$ and the corresponding trajectory $\gamma_u^{t_0,x}$ are called {\it{optimal}}\, from \,$x$\, to $y$\, if \,$u$ minimizes \eqref{controlquantity}. Note that  it is not clear if for every $x,\,y \in \mathbb{R}^d$, there is always at least one admissible trajectory joining them. To avoid this situation, we assume the following extra condition (see also \cite{Cannarsa}):\smallskip

(H3)\quad The convex hull of $k(t,x,U)$ contains an open neighborhood of the origin,\\
\hspace*{4em} i.e., $\exists \,\alpha >0\,\,\,\mbox{such that}\, \,\forall\, x \in \mathbb{R}^d,\,\exists\,u \in U:  k(t,x,u) \cdot v \leq -\alpha|v|$\\
\hspace*{4em} $\mbox{for all}\,v \in \mathbb{R}^d,\,t\in\mathbb{R}^+.$
\begin{lemma} \label{admissible trajectory}
Let \,$\Omega \subset \mathbb{R}^d$ be a compact domain. Then, there exists a constant $C$ depending only on $\alpha,\kappa$ and $\mbox{diam}(\Omega)$ such that, for all $t_0\in \mathbb{R}^+,\,x,\,y \in \Omega$, there is some control $u$ such that
$\tau_u^{t_0,x,y} \leq C |x-y|.$
\end{lemma}
\begin{proof}
Fix $x,\,y \in \Omega$. By (H3), there exists a constant control $u_0$ such that $k(t,x,u_0) \cdot (x-y) \leq -\alpha|x-y|$, for every $t \in \mathbb{R}^+$. Set $x_1=\gamma_{u_0}^{t_0,x}(t_1)$, where $t_1:=t_0 +\delta |x-y|$ and $\delta >0$ is to be chosen later. But again, there is some constant control $u_1$ such that $k(t,x_1,u_1) \cdot (x_1-y) \leq -\alpha|x_1-y|$, for every $t \in \mathbb{R}^+$. Set $x_2=\gamma_{u_1}^{t_1,x_1}(t_2)$, where $t_2:=t_1 + \delta|x_1-y|$. In this way, we get three sequences $(u_k)_k,\,(t_k)_k$ and $(x_k)_k$ such that:
$k(t,x_k,u_k) \cdot (x_k-y) \leq -\alpha|x_k-y|$, for every $t \in \mathbb{R}^+$, and $x_{k+1}=\gamma_{u_k}^{t_k,x_k}(t_{k+1})$, where $t_{k+1}=t_0 + \sum_{i=0}^k\delta |x_i-y|$. Using (H1) \& (H2), for every $t \in \left[t_k,t_{k+1}\right]$, we have:
\begin{eqnarray*}
(|\gamma_{u_k}^{t_k,x_k}(t) - y|^2/2)^\prime&=& k(t,\gamma_{u_k}^{t_k,x_k}(t),u_k) \cdot (\gamma_{u_k}^{t_k,x_k}(t) - y)\\
&=&k(t,x_k,u_k) \cdot (x_k - y) + k(t,\gamma_{u_k}^{t_k,x_k}(t),u_k) \cdot (\gamma_{u_k}^{t_k,x_k}(t) - x_k)\\
&& \qquad\qquad+\,\,\, (k(t,\gamma_{u_k}^{t_k,x_k}(t),u_k) - k(t,x_k,u_k)) \cdot (x_k - y)\\
&\leq& -\alpha |x_k - y| + \kappa |\gamma_{u_k}^{t_k,x_k}(t) - x_k| + C|\gamma_{u_k}^{t_k,x_k}(t) - x_k| |x_k - y|\\
&\leq& -\alpha |x_k - y| + \kappa(\kappa + C|x_k - y|)(t_{k+1} - t_k)\\
&=& (-\alpha  + \delta \kappa (\kappa + C|x_k - y|))|x_k - y|.
\end{eqnarray*}
 Then, we get
$$|x_{k+1} - y|^2 \leq (1 + 2\delta(-\alpha +\delta\kappa(\kappa + C\,\mbox{diam}(\Omega))))|x_k - y|^2.$$
Set 
$$\theta:=\sqrt{1 + 2\delta(-\alpha +\delta\kappa(\kappa + C\,\mbox{diam}(\Omega)))} < 1.$$\\
We then have $|{x_k} - y| \leq \theta^k |x - y|.$
Now, define
$$\bar{u}(t):=\begin{cases}
u_k \,\,\,\,\,\mbox{if}\,\,\,\,\,t \in \left[t_k,t_{k+1}\right],\,k \in \mathbb{N},\\
0 \,\,\,\,\,\,\,\mbox{else}.
\end{cases}$$
Observe that
$$\bar{t}:=\lim_{k \to +\infty} t_k=t_0 +\lim_{k \to +\infty}  \sum_{i=0}^{k-1} \delta |x_i -y| \leq t_0 +\lim_{k \to +\infty} \sum_{i=0}^{k-1} \delta \theta^i |x -y|=t_0 + \frac{\delta}{1- \theta}\,|x-y|.$$
On the other hand, we see easily that $\gamma^{t_0,x}_{\bar{u}}(\bar{t})=\lim_{k \to \infty} \gamma^{t_0,x}_{\bar{u}}(t_k)=\lim_{k \to \infty} x_k=y$. Consequently, we get
$$\tau^{t_0,x,y}_{\bar{u}} \leq \frac{\delta}{1- \theta}\,|x-y|,$$
and we are done. $\qedhere$
\end{proof}
We now introduce assumptions on the Lagrangian $L$ that will be needed in the sequel:\smallskip

(H4) \quad There exist two constants $\beta_1,\,\beta_2 >0$ such that 
$$\beta_1 \geq L(t,x,u) \geq \beta_2,\,\,\mbox{for all}\,\,(t,x,u) \in \mathbb{R}^+ \times \mathbb{R}^d \times U.
$$

(H5) \quad There exists a constant $C$ such that 
$$|L(t,x,u) - L(t,x^\prime,u)| \leq C|x - x^\prime|,\,\, \hbox{for all $x,\,x^\prime \in \mathbb{R}^d,\,t \in \mathbb{R}^+$ and $u \in U$}.
$$

(H6) \quad For any $(t,x) \in \mathbb{R}^+ \times \mathbb{R}^d$, the set 
$$\{ (v,r) \in \mathbb{R}^{d+1} \,:\, \exists \,u \in U \,\,\mbox{s.t.}\,\,k(t,x,u)=v,\, L(t,x,u) \leq r\} \hbox{ is convex}.
$$ 
Under assumptions (H0), (H1), (H2), (H3), (H4) and (H6), we have the following existence result. The proof is essentially based on some arguments used in \cite{Cannarsa}.

\begin{proposition} \label{PropExistOptim}
For every $t_0 \in \mathbb R^+$ and for all $x,\,y \in \Omega$, there exists an optimal control $u$ that minimizes the cost $J^{t_0, x, y}$ defined in \eqref{controlquantity}.
\end{proposition}

\begin{proof} 
Let $(u_k)_k$ be a minimizing sequence. For simplicity, set $\tau_k:=\tau^{t_0,x,y}_{u_k}$ and $\gamma_k:=\gamma^{t_0,x}_{u_k}$. By (H4), it is clear that, up to extracting subsequences, $\tau_k\rightarrow \bar{\tau}$ and $\gamma_k \rightarrow \gamma$ with $\gamma(t_0 +\bar{\tau})=y$. We first prove that $\gamma$ is an admissible trajectory. Fix $t \geq t_0$. By (H2), for all $s \geq t_0$, we have
\begin{eqnarray*}
|k(s,\gamma_k(s),u_k(s)) - k(t,\gamma(t),u_k(s))| &\leq& C(|s-t| +|\gamma_k(s) - \gamma(t)|)\\
&\leq& C(|s-t| + |\gamma_k(s) - \gamma_k(t)| + |\gamma_k(t) - \gamma(t)|)\\
&\leq& C(|s - t| + ||\gamma_k - \gamma||_\infty).
\end{eqnarray*}
Fix $\varepsilon >0$. For $|s-t| << \varepsilon$ and $k$ large enough, this implies that
$$k(s,\gamma_k(s),u_k(s)) \in k(t,\gamma(t),U) + \bar{B}(0,\varepsilon).$$
By using (H6), we get
$$\frac{\gamma_k(t+h) - \gamma_k(t)}{h}=\frac{1}{h} \int_t^{t+h} k(s,\gamma_k(s),u_k(s))\,\mathrm{d}s \,\,\in \,\,k(t,\gamma(t),U) +\bar{B}(0,\varepsilon).$$
Letting $k \to \infty$, we obtain
$$\frac{\gamma(t+h) - \gamma(t)}{h}\,\,\in \,\,k(t,\gamma(t),U) + \bar{B}(0,\varepsilon).$$
Letting $h \to 0$, we get
$$\dot\gamma(t)\,\,\in \,\,k(t,\gamma(t),U) + \bar{B}(0,\varepsilon),\,\,\mbox{for a.e.}\,\,t \geq t_0.$$
Since $\varepsilon >0$ is arbitrary, we infer that $\dot\gamma(t)\,\,\in \,\,k(t,\gamma(t),U),\,\,\mbox{for a.e.}\,\,t \geq t_0$. This is sufficient to say that $\gamma$ is an admissible trajectory, i.e. there is some control $u$ such that $\gamma=\gamma^{t_0,x}_u$.

 Now consider the augmented control system $(\tilde{U},\tilde{k})$ with state space $\mathbb{R}^{d+1}$, where $\tilde{U}=U \times \left[0,1\right]$ and $\tilde{k}(t,\tilde{x},\tilde{u})=(k(t,x,u),\beta_2 {\bf{u}} + L(t,x,u)(1-{\bf{u}}))$, where \,$\tilde{x}=(x,{\bf{x}}) \in \mathbb{R}^{d+1}$\, and \,$\tilde{u}=(u,{\bf{u}}) \in \tilde{U}$. Recalling (H6), one can check easily that the set $\tilde{k}(\tilde{x},\tilde{U})$ is convex, for all $\tilde{x} \in \mathbb{R}^{d+1}$. Let $\tilde{\gamma}_k$ be the trajectory which corresponds to initial condition $(x,0)$ at time $t_0$ and to the controls $(u_k,0)$. Then, we have $\tilde{\gamma}_k=(\gamma_k,{\bf{x}}_k)$, where 
$${\bf{x}}_k(t)=\int_{t_0}^t L(s,\gamma_k(s),u_k(s))\,\mathrm{d}s,\,\,\mbox{for every}\,\,\,t \geq t_0.$$
As \,$\tilde{k}(\tilde{x},\tilde{U})$\, is convex, for all $\tilde{x} \in \mathbb{R}^{d+1}$, then one can prove as before that there is a control $\tilde{u}=(u,{\bf{u}})$ such that $(\gamma_k,{\bf{x}}_k) \rightarrow (\gamma,{\bf{x}})$, where $(\gamma,{\bf{x}})$ is an admissible trajectory which corresponds to initial condition $(x,0)$ at time $t_0$ and to the control $\tilde{u}$. Yet, one has
$$\int_{t_0} ^t L(s,\gamma(s),u(s))\,\mathrm{d}s \leq  \int_{t_0} ^t {\bf{x}}^\prime(s)\,\mathrm{d}s={\bf{x}}(t)=\lim_{k \to \infty} {\bf{x}}_k (t)  =\lim_{k \to \infty}\int_{t_0}^t L(s,\gamma_k(s),u_k(s))\,\mathrm{d}s.$$
Recalling that $\tau^{t_0,x,y}_u \leq \bar{\tau}$ and $\tau_k \to \bar{\tau}$ and using (H4), we infer that
\begin{align*}
\int_{t_0} ^{t_0 +\tau^{t_0,x,y}_u} L(s,\gamma(s),u(s))\,\mathrm{d}s \leq&\ \int_{t_0} ^{t_0+\bar{\tau}} L(s,\gamma(s),u(s))\,\mathrm{d}s\\
 \leq&\  \lim_{k \to \infty}\int_{t_0}^{t_0+\bar{\tau}} L(s,\gamma_k(s),u_k(s))\,\mathrm{d}s\\
\leq&\  \lim_{k \to \infty}\int_{t_0}^{t_0+\tau_k} L(s,\gamma_k(s),u_k(s))\,\mathrm{d}s.
\end{align*}
This completes the proof that \,$\gamma$\, is an optimal trajectory from \,$x$\, to \,$y$ and  that $u$ is the corresponding optimal control. $\qedhere$
\end{proof}
Next, we want to give a result about the Lipschitz continuity of the transport cost. This result is standard but we prove it for completeness (see also \cite{Cannarsa} for similar results). Set
$$c(x,y):= \min\left\{J^{0,x,y}(u)\,;\,u:\mathbb{R}^+ \mapsto U\right\}.$$
\begin{proposition} \label{Lipschitz regularity of the transport cost}
Under Assumptions (H0)-(H6), the transport cost $c$ is Lipschitz continuous in $\Omega \times \Omega$.
\end{proposition}
\begin{proof}
Let $x,\,y,\,y^\prime \in \Omega$. Suppose that $c(x,y^\prime) > c(x,y)$ and let $u$ be an optimal control from $x$ to $y$ starting at time $0$. Set $\tau^\star=\tau^{0,x,y}_u$ and let $u^\star$ be a control such that $\tau^{\tau^\star,y,y^\prime}_{u^\star} \leq C |y - y^\prime|$ (see Lemma \ref{admissible trajectory}). Set 
$$\tilde{u}(t):=\begin{cases}
u(t) \,\,\,\,\,\,\,\,\,\mbox{if}\,\,\,\,t \in \left[0,\tau^\star\right],\\
u^\star(t)\,\, \,\,\,\mbox{else}.
\end{cases}$$
 Then, we have
\begin{eqnarray*}
c(x,y^\prime) - c(x,y) &\leq& J^{0,x,y^\prime}(\tilde{u}) - J^{0,x,y}(u)\\
&=& \int_{\tau^\star}^{\tau^\star +\tau^{\tau^\star,y,y^\prime}_{u^\star}} L(s,\gamma^{\tau^\star,y}_{u^\star}(s),u^\star(s))\,\mathrm{d}s\\
& \leq& C |y - y^\prime|.
\end{eqnarray*}
On the other hand, let $x,\,x^\prime,\,y \in \Omega$. Suppose that $c(x^\prime,y) > c(x,y)$ and let $u$ be an optimal control from $x$ to $y$, at time $0$. There are two possibilities: $\tau^{0,x,y}_u \leq \tau^{0,x^\prime,y}_u$ or $\tau^{0,x,y}_u > \tau^{0,x^\prime,y}_u$. First, assume that $\tau^\star:=\tau^{0,x,y}_u \leq \tau^{0,x^\prime,y}_u$. Set $y^\star=\gamma^{0,x^\prime}_u(\tau^\star)$
and let $u^\star$ be a control such that $\tau^{\tau^\star,y^\star,y}_{u^\star} \leq C |y - y^\star|$. Set 
$$\tilde{u}(t):=\begin{cases}
u(t) \,\,\,\,\,\,\,\,\,\mbox{if}\,\,\,\,t \in \left[0,\tau^\star\right],\\
u^\star(t)\,\, \,\,\,\mbox{else}.
\end{cases}$$
 Then, we have
\begin{eqnarray*}
c(x^\prime,y) - c(x,y) &\leq& J^{0,x^\prime,y}(\tilde{u}) - J^{0,x,y}(u)\\
&\leq& \int_0^{\tau^\star} (L(s,\gamma^{0,x^\prime}_{u}(s),u(s)) - L(s,\gamma^{0,x}_{u}(s),u(s)))\,\mathrm{d}s\\
&& \qquad\qquad+\,\,  \int_{\tau^\star}^{\tau^\star +\tau^{\tau^\star,y^\star,y}_{u^\star} } L(s,\gamma^{0,x^\prime}_{\tilde{u}}(s),\tilde{u}(s))\,\mathrm{d}s\\
& \leq & \int_0^{\tau^\star} |\gamma^{0,x^\prime}_{u}(s) - \gamma^{0,x}_{u}(s)|\,\mathrm{d}s + C |y - y^\star|.
\end{eqnarray*}
But, 
\begin{eqnarray*}
|\gamma^{0,x^\prime}_{u}(s) - \gamma^{0,x}_{u}(s)|&=&\bigg|\int_0^s k(t,\gamma^{0,x^\prime}_{u}(t),u(t))\,\mathrm{d}t + x^\prime - \int_0^s k(t,\gamma^{0,x}_{u}(t),u(t))\,\mathrm{d}t - x\bigg|\\
&\leq& |x-x^\prime| + C\int_0^s |\gamma^{0,x^\prime}_{u}(t) - \gamma^{0,x}_{u}(t)|\,\mathrm{d}t.
\end{eqnarray*}
Hence,
$|\gamma^{0,x^\prime}_{u}(s) - \gamma^{0,x}_{u}(s)| \leq C |x - x^\prime|,$
and consequently, 
$c(x^\prime,y) - c(x,y) \leq C |x - x^\prime|$. The case where $\tau^{0,x,y}_u > \tau^{0,x^\prime,y}_u$ can be treated in a similar fashion. $\qedhere$
\end{proof}
We also have the following stability-optimality result:
\begin{proposition} \label{stability-optimality}
Let $(x_k)_k$ and $(y_k)_k$ be two sequences such that $x_k \to x$ and $y_k \to y$. If $\gamma_k$ is an optimal trajectory from $x_k$ to $y_k$, for every $k \in \mathbb{N}$, then, up to a subsequence, $\gamma_k \rightarrow \gamma$, where $\gamma$ is an optimal trajectory from $x$ to $y$.
\end{proposition}
\begin{proof}
First, it is easy to see that there is some $\gamma$ such that $\gamma_k \rightarrow \gamma$. Set $\tau_k:=\tau^{0,x_k,y_k}_{u_k}$, where $u_k$ is the corresponding optimal control. Using (H4) \& Lemma \ref{admissible trajectory}, we infer that $\tau_k \to \bar{\tau}$. Recalling the proof of Proposition \ref{PropExistOptim}, this $\gamma$ is an admissible trajectory from $x$ to $y$ (there is some control $u$ such that $\gamma=\gamma_u^{0,x}$ with $\tau^{0,x,y}_u \leq \bar{\tau}$). Recalling Proposition \ref{PropExistOptim} and using Proposition \ref{Lipschitz regularity of the transport cost}, one can show the following:
\begin{eqnarray*}
\int_{0} ^{\tau^{0,x,y}_u} L(s,\gamma(s),u(s))\,\mathrm{d}s \leq \int_{0} ^{\bar{\tau}} L(s,\gamma(s),u(s))\,\mathrm{d}s &\leq & \lim_{k \to \infty}\int_{0}^{\tau_k} L(s,\gamma_k(s),u_k(s))\,\mathrm{d}s\\
&=&c(x,y).
\end{eqnarray*}
This implies that $\gamma$ is optimal from $x$ to $y$. Moreover, by (H4), we also get that $\bar{\tau}=\tau^{0,x,y}_u$ and so, $\tau^{0,x_k,y_k}_{u_k} \to \tau^{0,x,y}_u$. $\qedhere$ 
\end{proof}
Now consider the following multi-valued map $\Gamma$\,: 
$$ (x,y) \mapsto \Gamma(x,y):=\{(\gamma^{0,x,y}_u,u,\tau^{0,x,y}_u) \,:\,u\,\,\mbox{is an optimal control from}\,\,\,x \,\,\mbox{to}\,\,y\}.$$
The following is a direct consequence of Proposition \ref{stability-optimality} and a standard selection theorem (e.g., see \cite{13}). 

\begin{corollary} \label{closed graph}
The multi-valued map \,$\Gamma$ has a closed graph and therefore has a Borel selector function which will be denoted, in the sequel, by $\Gamma^s$.
\end{corollary}
\noindent We shall set the notation
$$(\gamma^{x,y},u^{x,y},\tau^{x,y}):=\Gamma^s(x,y),\,\,\mbox{for all}\,\,\,(x,y)\in\Omega \times \Omega,$$
which will eventually be used in Section \ref{Sec. 4}.

\section{Kantorovich duality with boundary tariffs}
 \label{Sec. 3}
In this section, we analyze Problem \eqref{Kantorovich with boundary costs}, its dual principle, and a decomposition of the optimal transport plan into three components according to whether they never visit the boundary, whether they get to it, or if they come from it. \\
 Set
$$\mathcal{P}_0 (\mu^+,\mu^-):=\bigg\{\pi \in \mathcal{M}^+(\Omega\times\Omega)\,:\,(\pi_x)_{\,|\accentset{\circ}{\Omega}}=\mu^+,\,(\pi_y)_{\,|\accentset{\circ}{\Omega}} =\mu^-\bigg\},$$
and consider the following problems:
\small{
 \begin{equation}\label{Problem 1}
{\mathcal T}_1(\mu^\pm;\psi^\pm):= \inf\bigg\{\int_{\Omega \times \Omega}c\,\mathrm{d}\pi
  \,+\,\int_{\partial\Omega}\psi^+\,\mathrm{d}\pi_x \,-\,\int_{\partial\Omega}\psi^-\,\mathrm{d}\pi_y
   \,:\,\pi \in \mathcal{P}_0(\mu^+,\mu^-) \bigg\}
 \end{equation}
 }
\hspace{-0.5em} and
\begin{equation}   \label{Problem 10}
 {\mathcal D}_1(\mu^\pm; \psi^\pm):=   \sup_{\phi^\pm \in C(\Omega)}\left\{\int_\Omega \phi^- \,\mathrm{d} \mu^- - \int_\Omega \phi^+ \,\mathrm{d}\mu^+
\,:\, \begin{array}{l}\psi^- \leq \phi^-\leq \phi^+ \leq \psi^+\; \mbox{on}\; \partial\Omega,\\
(-\phi^+) \oplus \phi^- \leq c \end{array}\right\}.
\end{equation}
In this section, we shall assume that the cost $c$ is Lipschitz continuous on $\Omega\times \Omega$ and that the two boundary costs $\psi^+$ and $\psi^-$ are in $C(\partial\Omega)$ and satisfy the inequality: 
 \begin{equation} \label{g_1g_2.0}
 \psi^-(y)-\psi^+(x)\leq c(x,y)\,\,\;\mbox{for all}\;\,x,\,y \in \partial\Omega.
 \end{equation}
Under these assumptions, 
we have the following result.
\begin{proposition} \label{existence of a minimizer}
The infimum in Problem \eqref{Problem 1} is attained, and the following duality formula holds: 
\begin{equation}\label{T1:D1}
{\mathcal T}_1 (\mu^+, \mu^-; \psi^+, \psi^-)={\mathcal D}_1 (\mu^+, \mu^-; \psi^+, \psi^-).
\end{equation}
   \end{proposition}
  \begin{proof} 
    Set 
   $$P(\pi):=\int_{\Omega\times\Omega} c\,\mathrm{d} \pi\; +\int_{\partial\Omega}\psi^+\,\mathrm{d}\pi_x\, - \int_{\partial\Omega}\psi^-\,\mathrm{d}\pi_y,\;\;\forall\;\pi \in \mathcal{M}(\Omega\times\Omega).$$
  $P$ is then continuous on $\mathcal{P}_0 (\mu^+,\mu^-)$ with respect to the weak$^*$ convergence of measures. Indeed, if $(\pi_n)_n$ is a sequence in $\mathcal{P}_0(\mu^+,\mu^-)$ such that $\pi_n{\rightharpoonup}\,\pi$, then, for every $n$,
 there exists $\xi_n^\pm \in \mathcal{M}^+(\partial\Omega)$ such that 
 $$\pi_x^n=\mu^+ + \xi_n^+\,,\;\;\pi_y^n=\mu^- + \xi_n^-\quad \hbox{
 and \,\,
 $ \xi_n^\pm \rightharpoonup \xi^\pm,$}$$
 where $\,\pi_x=\mu^+ + \xi^+\,\,$ and $\,\,\pi_y=\mu^- + \xi^-$.
As $\,\psi^+$ and \,$\psi^-$ are continuous on $\partial\Omega$, it follows that 
$P(\pi_n) \rightarrow P(\pi).$

On the other hand, we observe that if \,$\pi \in \mathcal{P}_0 (\mu^{+},\mu^{-})$\, and \,$\tilde{\pi}:=\pi_{\,|(\partial\Omega\times\partial\Omega)^c}$, then \,$\tilde{\pi}$\, also belongs to \,$\mathcal{P}_0 (\mu^{+},\mu^{-})$. In addition, we have
\begin{eqnarray*}
\int_{\Omega\times\Omega} c\,\mathrm{d} \pi \,+\int_{\partial\Omega}\psi^+\,\mathrm{d}\pi_x\, - \int_{\partial\Omega}\psi^-\,\mathrm{d}\pi_y&=&
\int_{\partial{\Omega}\times\partial\Omega} (c(x,y) + \psi^+(x) - \psi^-(y))\,\mathrm{d} \pi(x,y)\\
&&+\, \int_{(\partial{\Omega}\times\partial{\Omega})^c}c\,\mathrm{d} \pi \,+ \int_{\partial{\Omega}\times{\Omega}^{\circ}} \psi^+(x)\,\mathrm{d} \pi(x,y)
\,\\
&& \qquad- \int_{{\Omega}^{\circ}\times\partial{\Omega}} \psi^-(y)\,\mathrm{d} \pi(x,y).
\end{eqnarray*}
By \eqref{g_1g_2.0}, we get that  
$$  \int_{\Omega \times \Omega} c\,\mathrm{d} \pi\; +\int_{\partial\Omega}\psi^+\,\mathrm{d}\pi_x\, - \int_{\partial\Omega}\psi^-\,\mathrm{d}\pi_y  \geq  \int_{\Omega\times\Omega} c\,\mathrm{d}\tilde{\pi} \,+\int_{\partial\Omega}\psi^+\,\mathrm{d}\tilde{\pi}_x \,-\int_{\partial\Omega}\psi^-\,\mathrm{d}\tilde{\pi}_y.$$\\
Now take a minimizing sequence $(\pi_n)_n\subset\mathcal{P}_0(\mu^+,\mu^-)$ for \eqref{Problem 1}. We can suppose that $\pi_n (\partial\Omega\times\partial\Omega)=0.$ 
In this case, we get 
\begin{eqnarray*}
\pi_n(\Omega\times\Omega) &\leq  & \pi_n({\Omega}^0\times\Omega)+
  \pi_n(\Omega\times{\Omega}^0)
\\ 
&=& \mu^+(\Omega) + \mu^-(\Omega).
\end{eqnarray*}
 Hence, there exist a subsequence $ (\pi_{n_k})_{n_k}$ and a plan $\pi \in \mathcal{P}_0 (\mu^+,\mu^-) $ such that
 $ \pi_{n_k} {\rightharpoonup}\,\pi$. But, the continuity of the total cost \,$P$\, implies that this \,$\pi$\, is a minimizer for \eqref{Problem 1}.\\
 
  We now establish the duality formula. 
 A proof of the duality formula for a general Lipschitz cost $c$ 
based on the Fenchel-Rockafellar duality was given in \cite{7}. Here, we give an alternative proof based on  a perturbative argument used in \cite{Dweik,8}. Define the  functional $H: C(\partial\Omega)\times C(\partial\Omega) \mapsto \mathbb{R} \cup \{+\infty\}$ as follows:
\begin{eqnarray*} H(h^+,h^-)
 =-\sup\left\{\int_\Omega \varphi^- \mathrm{d} \mu^- - \int_\Omega \varphi^+ \mathrm{d}\mu^+: \begin{array}{l} \psi^- + h^- \leq \varphi^-\leq \varphi^+ \leq \psi^+ - h^+\;\mbox{on}\;\partial\Omega,\\
 (-\varphi^+) \oplus \varphi^- \leq c\end{array}\right\}.
 \end{eqnarray*}\\
Note that $H(h^+,h^-) \in \mathbb{R} \cup \{+\infty\}$. This follows immediately from the fact that for a maximizing sequence $(\varphi_k^+,\varphi_k^-)$, we can always assume them to share the same Lipschitz constant as $c$. In fact, if we replace \,$\varphi_k^-$\, by \,$\phi_k^-$ where \,$\phi_k^-(y) := \min\{c(x,y) + \varphi_k^+(x)\,:\,x \in \Omega\}$, for every $y \in \Omega$, then the constraints are preserved and the integrals increased. We can then assume that they are uniformly bounded 
in such a way that Ascoli-Arzel\`a's Theorem applies. Next, we show that $H$ is convex and lower semi-continuous.\\

  For the convexity\,: take \,$t \in (0,1)$\, and \,$(h_0^+,h_0^-),\,(h_1^+,h_1^-)\in C(\partial\Omega)\times C(\partial\Omega)$\,
 and, let \,$(\varphi_0^+,\varphi_0^-)$ and $(\varphi_1^+,\varphi_1^-)$ be their optimal potentials. Set $$ h_t^+:=(1-t)h_0^+ + t h_1^+,\,\,h_t^-:=(1-t)h_0^-+th_1^- $$ 
 and $$\varphi_t^+:=(1-t)\varphi_0^+ + t\varphi_1^+,\,\,\varphi_t^-:=(1-t)\varphi_0^- + t\varphi_1^-.$$
As 
$$ \psi^- + h_0^- \leq \varphi_0^-\ \,\,\mbox{and}\,\,\  \psi^- + h_1^- \leq \varphi_1^-\,\,\,\,\,\mbox{on}\,\,\,\partial\Omega,
$$
$$ \varphi_0^+ \leq  \psi^+ - h_0^+\ \,\,\mbox{and}\,\,\   \varphi_1^+ \leq   \psi^+ - h_1^+\,\,\,\,\,\mbox{on}\,\,\,\partial\Omega,
$$
and
$$(-\varphi_0^+) \oplus \varphi_0^- \leq c\,\,\;\mbox{and}\;\,\,(-\varphi_1^+) \oplus \varphi_1^- \leq c, $$ 
then 
$$
\psi^- + h_t^-   \leq \varphi_t^-\leq \varphi_t^+ \leq \psi^+ -  h_t^+\,\,\,\,\mbox{on}\;\, \partial\Omega, \,\,\;\,\mbox{and}\,\,\;(-\varphi_t^+) \oplus \varphi_t^- \leq c.$$ 
As a consequence of that, we infer that $(\varphi_t^+,\varphi_t^-)$ is admissible in the max defining $-H(h_t^+,h_t^-)\,$ and then, 
 $$ H(h_t^+,h_t^-)\leq \int_\Omega\varphi_t^+\,\mathrm{d}\mu^+ - \int_\Omega \varphi_t^- \,\mathrm{d}\mu^-
 =(1-t)H(h_0^+,h_0^-)+tH(h_1^+,h_1^-). $$
 
 For the semi-continuity$\,$: take $h_k^+\rightarrow h^+$ and $h_k^- \rightarrow h^-$  uniformly on $\partial\Omega$. Let $(h_{k_i}^+,h_{k_i}^-)_{k_i}$ be a subsequence such that $\liminf_k H(h_k^+,h_k^-)=\lim_{k_i}H(h_{k_i}^+,h_{k_i}^-)$ (for simplicity of notation, we still denote this subsequence by $(h_k^+,h_k^-)_k$) and let $(\varphi_k^+,\varphi_k^-)_k$ be
  their corresponding optimal potentials. As $\varphi_k^+,\,\varphi_k^-$ have the same Lipschitz constant as the transport cost $c$ and $(h_k^+)_k,\,(h_k^-)_k$ are equibounded, then, by Ascoli-Arzel\`a's Theorem, there are two continuous functions $\varphi^+,\,\varphi^-$ and a subsequence $(\varphi_k^+,\varphi_k^-)_{k}$ such that $\varphi_{k}^+\rightarrow\varphi^+$ and $\varphi_{k}^- \rightarrow \varphi^-$ uniformly in $\Omega$. As 
  $$\psi^- + h_k^-  \leq \varphi_k^-\leq \varphi_{k}^+ \leq \psi^+ - h_k^+  \,\;\,\mbox{on}\;\, \partial\Omega, \,\,\;\;\mbox{and}\,\,\,\;(-\varphi_k^+) \oplus \varphi_k^- \leq c, 
  $$
 then 
 $$  
 \psi^- + h^-\leq \varphi^-  \leq \varphi^+ \leq   \psi^+ - h^+\,\;\,\mbox{on}\,\; \partial\Omega, \;\,\,\;\mbox{and}\,\,\,\;(-\varphi^+) \oplus \varphi^- \leq c.
 $$ 
 Consequently, the pair $(\varphi^+,\varphi^-)$ is admissible in the max defining $ -H(h^+,h^-)$ and so, one has
 $$ H(h^+,h^-)\leq \int_\Omega\varphi^+ \,\mathrm{d}\mu^+ -\int_\Omega\varphi^- \,\mathrm{d}\mu^- 
 =\liminf_{k} H(h_{k}^+,h_{k}^-).$$ 
Since $H$ is convex and lower semi-continuous, it is equal to its double Legendre transform, i.e., $ H^{\star\star}=H $, and in particular, $H^{\star \star}(0,0)=H(0,0)$. By its definition, we have  
$$
H(0,0)=-\mathcal{D}_1(\mu^+,\mu^-;\psi^+,\psi^-).
$$ 
We now compute $H^{\star\star}(0,0)$. Take $\xi^+,\,\xi^-$ in $\mathcal{M}(\partial\Omega)$, then we have 
\begin{eqnarray*}
H^{\star}(\xi^+,\xi^-)=  \sup_{h^+,\,h^- \,\in \,C(\partial\Omega)
}\bigg\{\int_{\partial\Omega} h^+\,\mathrm{d}\xi^+ + \int_{\partial\Omega} h^-\,\mathrm{d}\xi^- - H(h^+,h^-)\bigg\},
\end{eqnarray*} 
which is equal to:
\begin{eqnarray*}
\sup_{(\varphi^\pm,h^\pm)}\left\{\int_{\partial\Omega} h^+ \mathrm{d}\xi^+ + \int_{\partial\Omega} h^- \mathrm{d}\xi^- - \int_\Omega \varphi^+ \mathrm{d}\mu^+ + \int_\Omega \varphi^- \mathrm{d}\mu^-  : \begin{array}{l}
\psi^- + h^- \leq \varphi^- \,\,\,\mbox{on}\,\,\,\partial\Omega,\\
\varphi^+ \leq\psi^+ - h^+\,\,\,\mbox{on}\,\,\,\partial\Omega,\\
(-\varphi^+) \oplus  \varphi^- \leq c \end{array}\right\}.
\end{eqnarray*}
If $\xi^+\notin \mathcal{M}^+(\partial\Omega) $, then there exists $h_0^+ \in C(\partial\Omega)$ 
 such that $h_0^+\geq 0$ and $\int_{\partial\Omega} h_0^+\,\mathrm{d}\xi^+ <0, 
$ and therefore
$$H^{\star}(\xi^+,\xi^-)\geq\;-k\int_{\partial\Omega} h_0^+\,\mathrm{d}\xi^+ +\int_{\partial\Omega} \psi^+\,\mathrm{d}\xi^+ -\int_{\partial\Omega} \psi^-\,\mathrm{d}\xi^-\;\underset{k \to + \infty}{\longrightarrow}  \;+\infty.$$ 
Similarly if $\xi^-\notin \mathcal{M}^+(\partial\Omega).$ Now suppose that $\xi^\pm \in \mathcal{M}^+(\partial\Omega) $. As $\psi^- + h^-\leq \varphi^- \ \mbox{and}\ \varphi^+ \leq  \psi^+ - h^+\;\mbox{on}\;\partial\Omega,$ we should choose the largest possible $h^+$ and $h^-$, i.e. $\,h^+(x)= \psi^+(x) - \varphi^+(x)\,$ and $\,h^-(y)=\varphi^-(y) - \psi^-(y)$, for all $x,\,y \in\partial\Omega$. Hence, we get
\begin{eqnarray*}
 H^{\star}(\xi^+,\xi^-)&=&\sup\bigg\{\int_{\Omega} \varphi^- \mathrm{d}(\mu^- + \xi^-) - \int_{\Omega} \varphi^+ \mathrm{d}(\mu^+ + \xi^+) : (-\varphi^+) \oplus \varphi^- \leq c \bigg\}\\
 &&  \qquad\qquad\qquad+\,\,\, \int_{\partial\Omega} \psi^+ \mathrm{d}\xi^+ - \int_{\partial\Omega} \psi^- \mathrm{d}\xi^-,
 \end{eqnarray*} 
and therefore,
\begin{eqnarray*}
 H^{\star}(\xi^+,\xi^-)&=& \inf\left\{\int_{\Omega\times\Omega}c(x,y)\,\mathrm{d}\pi
\,:\,\pi \in \mathcal{P}(\mu^+ + \xi^+,\mu^- + \xi^-)\right\}\\
&&\qquad\qquad\qquad\qquad\qquad\qquad
+ \,\,\,\int_{\partial\Omega} \psi^+\,\mathrm{d}\xi^+ - \int_{\partial\Omega} \psi^-\,\mathrm{d}\xi^-\\ 
&=&\inf_{\pi \in \mathcal{P}(\mu^+ + \xi^+,\mu^-+\xi^-)}\left\{\int_{\Omega\times\Omega}c(x,y)\,\mathrm{d}\pi
+ \int_{\partial\Omega} \psi^+\,\mathrm{d}\pi_x
-\int_{\partial\Omega} \psi^-\,\mathrm{d}\pi_y \right\}.
\end{eqnarray*} 
Consequently,
\begin{eqnarray*}
H^{\star\star}(0,0)&=&\sup\bigg\{-H^{\star}(\xi^+,\xi^-)\,:\,\xi^+,\,\xi^-\in\mathcal{M}^+(\partial\Omega)\bigg\}\\
&=&-\inf \left\{\int_{\Omega\times\Omega}c(x,y)\,\mathrm{d}\pi +
\int_{\partial\Omega}\psi^+\,\mathrm{d}\pi_x -
\int_{\partial\Omega}\psi^-\,\mathrm{d}\pi_y\,:\,\pi\in \mathcal{P}_0 (\mu^+,\mu^-)\right\}.\qedhere
\end{eqnarray*}
\end{proof}
Our aim now is to decompose the transport problem \eqref{Problem 1} into three subproblems: the first one will be for transporting mass from the interior of the domain to the interior, making it a standard Monge-Kantorovich transport; the second transports mass from the interior to the boundary and the last one moves mass from the boundary to the interior. \\
To do that, we fix a minimizer $\pi$ for \eqref{Problem 1} and denote by $\xi^+$ and $\xi^-$ the two positive measures concentrated on the boundary of $\Omega$ such that $\pi_x=\mu^+  + \xi^+ $ and $ \pi_y=\mu^-+\xi^-$ (this means that $\xi^+,\,\xi^-$ encode the import/export masses).
 Set $$\pi^{ii}:=\pi_{\,|\Omega^{\circ}\times\Omega^{\circ}},\,
\pi^{ib}:=\pi_{\,|\Omega^{\circ}\times\partial\Omega},\,
\pi^{bi}:=\pi_{\,|\partial\Omega\times\Omega^{\circ}},\,
\pi^{bb}:=\pi_{\,|\partial\Omega\times\partial\Omega}$$ and
$$\nu^+:=\pi_x^{ib},\,
 \,\nu^-:=\pi_y^{bi}.$$ 
 Recalling the proof of Proposition \ref{existence of a minimizer}, we can assume that the boundary-boundary part $\pi^{bb}=0$ (this is always possible thanks to our assumption \eqref{g_1g_2}). Now, let $T^{ib}$ (resp., $T^{bi}$) be a Borel selector function of the following possibly multivalued set of minimizers
$$T^{ib}(x)\in \mbox{argmin}\left\{c(x,y) - \psi^-(y),\,y \in \partial\Omega\right\},\, \; \mbox{for all} \;x \in \Omega.$$ 
(resp., 
$$T^{bi}(y)\in \mbox{argmin}\left\{c(x,y) + \psi^+(x),\,x \in \partial\Omega\right\},\,\,\mbox{for all}\,\;y \in \Omega.)$$
Note that such a selector exists since their graph is closed and so one can use a selection theorem, such as in \cite{13}. 
The following proposition is straightforward.
 \begin{proposition}
Consider the following three transport problems:
\begin{equation} \label{(P1)}
\inf \left\{\int_{\Omega\times\Omega} c(x,y)\,\mathrm{d}\Lambda
\,:\,\Lambda \in \mathcal{P}(\mu^{+}-\nu^+,\mu^{-}-\nu^-)\right\},
\end{equation}
\begin{equation}\label{(P2)}
\inf \left\{\int_{\Omega\times\Omega} c(x,y)\,\mathrm{d}\Lambda\,-\,\int_{\partial{\Omega}}\psi^-\,\mathrm{d}\chi^-
\,:\,\Lambda\in\mathcal{P}(\nu^+,\chi^-),\,\,\spt(\chi^-) \subset \partial\Omega \right\},
\end{equation}
and
\begin{equation} \label{(P3)}
 \inf \left\{\int_{\Omega\times\Omega} c(x,y)\,\mathrm{d}\Lambda \,+\,
\int_{\partial{\Omega}}\psi^+\,\mathrm{d}\chi^+
\,:\,\Lambda \in \mathcal{P}(\chi^+,\nu^-),\,\,\spt(\chi^+) \subset \partial\Omega \right\}.
\end{equation}
Then, $\pi^{ii}$, $(\pi^{ib}, T^{ib}_{\#} \nu^+)$ and $(\pi^{bi},T^{bi}_{\#}\nu^-)$ solve \eqref{(P1)},\,\eqref{(P2)} and \eqref{(P3)}, respectively. 
Moreover, the optimal transport plan from the interior to the boundary $\pi^{ib}=(Id,T^{ib})_{\#} \nu^+$ and solves
 \begin{equation}\min\left\{\int_{\Omega\times\Omega}c(x,y)\,\mathrm{d}\pi
 \,:\,\pi \in \mathcal{P}(\nu^+,T^{ib}_{\#}\nu^+) \right\}.
  \end{equation} 
  Similarly, the optimal transport plan  from the boundary to the interior    $\pi^{bi}=(T^{bi},Id)_{\#} \nu^-$ and solves
 \begin{equation}
 \min\left\{\int_{\Omega\times\Omega}c(x,y)\,\mathrm{d}\pi
 \,:\,\pi \in \mathcal{P}(T^{bi}_{\#}\nu^-,\nu^-) \right\}.
 \end{equation}
 \end{proposition}
We finish this section by the following result.
\begin{proposition}
Let $(\varphi^+,\varphi^-)$ be a maximizer of the dual problem \eqref{dual with boundary costs} and let $\pi$ be an optimal transport plan for \eqref{Kantorovich with boundary costs}. Let $\xi^+$ and $\xi^-$ be the two positive measures on $\partial\Omega$ such that $\pi_x=\mu^+ + \xi^+$ and $\pi_y=\mu^- + \xi^-$. Then, we have the following
$$\varphi^+=\psi^+ \,\,\,\mbox{on}\,\,\,\spt(\xi^+),\,\,\varphi^-=\psi^- \,\,\,\mbox{on}\,\,\,\spt(\xi^-).$$ 
\end{proposition}
\begin{proof}
First, it is clear that $\pi$ solves 
\begin{equation} \label{Kantorovich with optimal boundary parts}
\min \left\{\int_{\Omega\times\Omega} c(x,y)\,\mathrm{d}\Lambda
\,:\,\Lambda \in \mathcal{P}(\mu^{+}+\xi^+,\mu^{-}+\xi^-)\right\}.
\end{equation}
Now, let $(\phi^+,\phi^-)$ be a maximizer for the dual problem of \eqref{Kantorovich with optimal boundary parts}:
\begin{equation} \label{dual with optimal boundary parts}
\sup\bigg\{\int_\Omega \phi^-\,\mathrm{d}[\mu^- + \xi^-] - \int_\Omega \phi^+\,\mathrm{d}[\mu^+ + \xi^+]\,\,:\,\,\phi^\pm \in C(\Omega),\,(- \phi^+) \oplus \phi^- \leq c\bigg\}.
\end{equation}
We have 
$$\int_\Omega \varphi^-\,\mathrm{d}[\mu^- + \xi^-] - \int_\Omega \varphi^+\,\mathrm{d}[\mu^+ + \xi^+] \leq \int_\Omega \phi^-\,\mathrm{d}[\mu^- + \xi^-] - \int_\Omega \phi^+\,\mathrm{d}[\mu^+ + \xi^+].$$
But, $\psi^- \leq \varphi^- \leq \varphi^+ \leq \psi^+$ on $\partial\Omega$. Then, we get
\begin{multline*}
\int_\Omega \varphi^-\,\mathrm{d}\mu^- + \int_\Omega \psi^-\,\mathrm{d}\xi^- - \int_\Omega \varphi^+\,\mathrm{d}\mu^+ - \int_\Omega \psi^+\,\mathrm{d}\xi^+ \\
\leq  \int_\Omega \phi^-\,\mathrm{d}[\mu^- + \xi^-] - \int_\Omega \phi^+\,\mathrm{d}[\mu^+ + \xi^+]=\int_{\Omega\times\Omega} c(x,y)\,\mathrm{d}\pi.
\end{multline*}
From the duality result ${\mathcal T}_1 (\mu^+, \mu^-; \psi^+, \psi^-)={\mathcal D}_1 (\mu^+, \mu^-; \psi^+, \psi^-)$ (see Proposition \ref{existence of a minimizer}), we conclude the proof (we note that this shows at the same time that the pair $(\varphi^+,\varphi^-)$ solves \eqref{dual with optimal boundary parts}). 
\end{proof}
\section{Eulerian formulation for transports involving tariffs}
\label{Sec. 4}
In this section, our aim is to find an equivalent Eulerian formulation for \eqref{Problem 1}.   We shall assume that all assumptions (H0)-(H6) hold and so, by Section \ref{Sec. 2}, for every $(x,y) \in \Omega \times \Omega$, we can associate an optimal trajectory $\gamma^{x,y}$, an optimal control $u^{x,y}$, and an end time $\tau^{x,y}$ to get from $x$ to $y$, in such a way that the three maps $(x,y) \mapsto \gamma^{x,y}$, $(x,y) \mapsto u^{x,y}$ and $(x,y) \mapsto \tau^{x,y}$ are  chosen to be measurable. 
We first introduce the following:
\begin{definition}
Say that $\rho: \mathbb{R}^+ \mapsto \mathcal{M}^+(\mathbb{R}^d \times U)$ is a density process and \,$\eta \in \mathcal{M}^+(\mathbb{R}^+ \times \mathbb{R}^d)$ is a stopping distribution between two given finite positive measures $\chi^+$ and $\chi^-$ on $\mathbb{R}^d$, if they satisfy the following:
\begin{equation} \label{stopping time}
\int_{\mathbb{R}^d} \int_{\mathbb{R}^+}\psi(x)\,\mathrm{d}\eta(t,x)=\int_{\mathbb{R}^d} \psi(x)\,\mathrm{d}\chi^+(x),\,\,\mbox{for all}\,\,\,\psi \in C_0(\mathbb{R}^d),
\end{equation}
and 
\begin{eqnarray} \label{process}
&&\int_{\mathbb{R}^+} \int_U \int_{\mathbb{R}^d} \left[\partial_t \xi(t,x) + k(t,x,u) \cdot \nabla \xi(t,x)\right]\,\mathrm{d}\rho_t(x,u)\,\mathrm{d}t\\
&&\qquad \qquad\qquad=\,\,
 \int_{\mathbb{R}^+}\int_{\mathbb{R}^d} \xi(t,x)\,\mathrm{d}\eta(t,x)- \int_{\mathbb{R}^d} \xi(0,x)\,\mathrm{d}\chi^-(x),\,\,\mbox{for all}\,\,\,\xi \in C_0^1(\mathbb{R}^+ \times{\mathbb{R}^d}).\notag
 \end{eqnarray}\\
The set of such pairs $(\rho,\eta)$ will be denoted by $\mathcal{E}(\chi^+,\chi^-)$.
\end{definition}
Consider now the transport subproblems \eqref{(P1)}, \eqref{(P2)} \& \eqref{(P3)}. Our goal is to construct, for each of these subproblems, an admissible pair of density process and stopping distributions $(\rho^{ii},\eta^{ii}) \in \mathcal{E}(\mu^+ - \nu^+,\mu^- - \nu^-)$ (resp. $(\rho^{ib},\eta^{ib}) \in \mathcal{E}(\nu^+,T^{ib}_{\#}\nu^+)$ and $(\rho^{bi},\eta^{bi}) \in \mathcal{E}(T^{bi}_{\#}\nu^-,\nu^-)$) from the optimal transport plan $\pi^{ii}$ (resp. $\pi^{ib}$ and $\pi^{bi}$). More precisely, we have the following:
\begin{lemma} \label{Eulerian ii}
There exists an admissible pair $(\rho^{ii},\eta^{ii}) \in \mathcal{E}(\mu^+ - \nu^+,\mu^- - \nu^-)$ such that $$\int_{\mathbb{R}^+} \int_U \int_{\mathbb{R}^d} L(t,x,u)\,\mathrm{d}\rho_t^{ii}(x,u)\,\mathrm{d}t=\int_{\Omega \times \Omega}c(x,y)\,\mathrm{d}\pi^{ii}(x,y).$$
\end{lemma}
\begin{proof}
Let $\pi^{ii}$ be an optimal transport plan between $\mu^+ - \nu^+$ and $\mu^- - \nu^-$ and define the pair $(\rho^{ii},\eta^{ii})$ as follows:
$$\int_{\mathbb{R}^+}\int_{\mathbb{R}^d} \xi(t,x)\,\mathrm{d}\eta^{ii}(t,x)=\int_{\Omega \times \Omega} \xi(\tau^{x,y},y)\,\mathrm{d}\pi^{ii}(x,y),\,\,\mbox{for all}\,\,\,\xi \in C_0(\mathbb{R}^+ \times \mathbb{R}^d),$$
and  for all $\zeta \in C_0(\mathbb{R}^d \times U)$, 
$$\int_U \int_{\mathbb{R}^d} \zeta(x,u)\,\mathrm{d}\rho_t^{ii}(x,u)=\int_{\{(x,y) \in \Omega \times \Omega \,\,:\,\,t \leq \tau^{x,y}\}} \zeta(\gamma^{x,y}(t),u^{x,y}(t))\,\mathrm{d}\pi^{ii}(x,y).
$$
One can easily check that the pair $(\rho^{ii},\eta^{ii}) \in \mathcal{E}(\mu^+ - \nu^+,\mu^- - \nu^-)$.
Indeed, for every $\psi \in C_0(\mathbb{R}^d)$\, and \,$\xi \in C_0^1(\mathbb{R}^+ \times \mathbb{R}^d)$, we have
\begin{equation*}
\int_{\mathbb{R}^+}\int_{\mathbb{R}^d}\psi(x)\,\mathrm{d}\eta^{ii}(t,x)=\int_{\Omega \times \Omega} \psi(y)\,\mathrm{d}\pi^{ii}(x,y)=\int_{\Omega} \psi(y)\,\mathrm{d}(\mu^- - \nu^-)(y),
\end{equation*}
and 
$$\int_{\mathbb{R}^+} \int_U \int_{\mathbb{R}^d} \left[\partial_t \xi(t,x) + k(t,x,u) \cdot \nabla \xi(t,x)\right]\,\mathrm{d}\rho_t^{ii}(x,u)\,\mathrm{d}t$$
$$=\int_{\Omega \times \Omega}\int_{0}^{\tau^{x,y}} \left[\partial_t \xi(t,\gamma^{x,y}(t)) + k(t,\gamma^{x,y}(t),u^{x,y}(t)) \cdot \nabla \xi(t,\gamma^{x,y}(t))\right]\,\mathrm{d}t\,\mathrm{d}\pi^{ii}(x,y)$$
$$=\int_{\Omega \times \Omega}\int_{0}^{\tau^{x,y}} \partial_t \left[\xi(t,\gamma^{x,y}(t))\right]\,\mathrm{d}t\,\mathrm{d}\pi^{ii}(x,y) =\int_{\Omega \times \Omega}\xi(\tau^{x,y},y)\,\mathrm{d}\pi^{ii}(x,y)- \int_{\Omega} \xi(0,x)\,\mathrm{d}(\mu^+ - \nu^+)(x)$$
$$= \int_{\mathbb{R}^+}\int_{\mathbb{R}^d} \xi(t,x)\,\mathrm{d}\eta^{ii}(t,x)- \int_{\Omega} \xi(0,x)\,\mathrm{d}(\mu^+ - \nu^+)(x).$$
Moreover, we have 
\begin{eqnarray*}
\int_{\mathbb{R}^+} \int_U \int_{\mathbb{R}^d} L(t,x,u)\,\mathrm{d}\rho_t^{ii}(x,u)\,\mathrm{d}t
&=&\int_{\Omega \times \Omega}\int_{0}^{\tau^{x,y}} L(t,\gamma^{x,y}(t),u^{x,y}(t))\,\mathrm{d}t\,\mathrm{d}\pi^{ii}(x,y)=\int_{\Omega \times \Omega}c\,\mathrm{d}\pi^{ii}. 
\end{eqnarray*}
\end{proof}
\begin{lemma} \label{Eulerian ib}
There exists an admissible pair $(\rho^{ib},\eta^{ib}) \in \mathcal{E}(\nu^+,T^{ib}_{\#}\nu^+)$ such that the following holds 
\begin{multline*}
\int_{\mathbb{R}^+} \int_U \int_{\mathbb{R}^d} L(t,x,u)\,\mathrm{d}\rho_t^{ib}(x,u)\,\mathrm{d}t - \int_{\partial\Omega}\int_{\mathbb{R}^+} \psi^-(x)\,\mathrm{d}\eta^{ib}(t,x)\\
=\int_{\Omega \times \Omega}c(x,y)\,\mathrm{d}\pi^{ib}(x,y) -  \int_{\partial\Omega} \psi^-(y)\,\mathrm{d}\left[T^{ib}_{\#}\nu^+\right](y).
\end{multline*}
\end{lemma}
\begin{proof}
Define the pair consisting of a density process and a stopping distribution $(\rho^{ib},\eta^{ib}) \in \mathcal{E}(\nu^+,T^{ib}_{\#}\nu^+)$ as follows:
$$\int_{\mathbb{R}^+}\int_{\mathbb{R}^d} \xi(t,x)\,\mathrm{d}\eta^{ib}(t,x)=\int_{\Omega} \xi(\tau^{x,T^{ib}(x)},T^{ib}(x))\,\mathrm{d}\nu^+(x),\,\,\mbox{for all}\,\,\,\xi \in C_0(\mathbb{R}^+ \times \mathbb{R}^d),$$\\
while, for all $\zeta \in C_0(\mathbb{R}^d \times U)$, 

$$\int_U \int_{\mathbb{R}^d} \zeta(x,u)\,\mathrm{d}\rho_t^{ib}(x,u)=\int_{\{x \in \Omega\,\,:\,\,t\leq \tau^{x,T^{ib}(x)}\}} \zeta(\gamma^{x,T^{ib}(x)}(t),u^{x,T^{ib}(x)}(t))\,\mathrm{d}\nu^+(x).
$$
We have, for every $\psi \in C_0(\mathbb{R}^d)$,

\begin{equation*}
\int_{\mathbb{R}^+}\int_{\mathbb{R}^d} \psi(x)\,\mathrm{d}\eta^{ib}(t,x)=\int_{\Omega} \psi(T^{ib}(x))\,\mathrm{d}\nu^+(x)=\int_{\Omega} \psi(y)\,\mathrm{d}\left[T^{ib}_{\#}\nu^+\right](y),
\end{equation*}\\
and, for every $\xi \in C_0^1(\mathbb{R}^+ \times \mathbb{R}^d)$, 
$$\int_{\mathbb{R}^+} \int_U \int_{\mathbb{R}^d} \left[\partial_t \xi(t,x) + k(t,x,u) \cdot \nabla \xi(t,x)\right]\,\mathrm{d}\rho_t^{ib}(x,u)\,\mathrm{d}t$$
$$=\int_{\Omega}\int_{0}^{\tau^{x,T^{ib}(x)}} \bigg(\partial_t \xi(t,\gamma^{x,T^{ib}(x)}(t)) + k(t,\gamma^{x,T^{ib}(x)}(t),u^{x,T^{ib}(x)}(t)) \cdot \nabla \xi(t,\gamma^{x,T^{ib}(x)}(t))\bigg)\,\mathrm{d}t\,\mathrm{d}\nu^+$$
$$=\int_{\mathbb{R}^+}\int_{\mathbb{R}^d} \xi(t,x)\,\mathrm{d}\eta^{ib}(t,x)- \int_{\Omega} \xi(0,x)\,\mathrm{d}\nu^+(x).$$\\
On the other hand, we have 
\begin{eqnarray*}
&&\int_{\mathbb{R}^+} \int_U \int_{\mathbb{R}^d} L(t,x,u)\,\mathrm{d}\rho_t^{ib}(x,u)\,\mathrm{d}t - \int_{\partial\Omega}\int_{\mathbb{R}^+} \psi^-(x)\,\mathrm{d}\eta^{ib}(t,x)\\
&&\qquad\qquad=\,\,\int_{\Omega}\int_{0}^{\tau^{x,T^{ib}(x)}} L(t,\gamma^{x,T^{ib}(x)}(t),u^{x,T^{ib}(x)}(t))\,\mathrm{d}t\,\mathrm{d}\nu^+(x) - \int_{\partial\Omega} \psi^-(y)\,\mathrm{d}\left[T^{ib}_{\#}\nu^+\right](y)\\\\
&&\qquad\qquad=\,\,\int_{\Omega}c(x,T^{ib}(x))\,\mathrm{d}\nu^+(x) - \int_{\partial\Omega} \psi^-(y)\,\mathrm{d}\left[T^{ib}_{\#}\nu^+\right](y). 
\end{eqnarray*}
This completes the proof. $\qedhere$
\end{proof}
\begin{lemma} \label{Eulerian bi}
There exists an admissible pair $(\rho^{bi},\eta^{bi}) \in \mathcal{E}(T^{bi}_{\#}\nu^-,\nu^-)$ such that we have
\begin{multline*}
\int_{\mathbb{R}^+} \int_U \int_{\mathbb{R}^d} L(t,x,u)\,\mathrm{d}\rho_t^{bi}(x,u)\,\mathrm{d}t + \int_{\partial\Omega}\int_{U} \psi^+(x)\,\mathrm{d}\rho_0^{bi}(x,u)\\
 =\,\,\,\int_{\Omega \times \Omega}c(x,y)\,\mathrm{d}\pi^{bi}(x,y) +  \int_{\partial\Omega} \psi^+(x)\,\mathrm{d}\left[T^{bi}_{\#}\nu^-\right](x).
\end{multline*}
\end{lemma}
\begin{proof}
Define the pair consisting of a density process and a stopping distribution $(\rho^{bi},\eta^{bi}) \in \mathcal{E}(T^{bi}_{\#}\nu^-,\nu^-)$ as follows:
$$\int_{\mathbb{R}^+}\int_{\mathbb{R}^d} \xi(t,x)\,\mathrm{d}\eta^{bi}(t,x)=\int_{\Omega} \xi(\tau^{T^{bi}(y),y},y)\,\mathrm{d}\nu^-(y),\,\,\mbox{for all}\,\,\,\xi \in C_0(\mathbb{R}^+ \times \mathbb{R}^d),$$\\
while, for all $\zeta \in C_0(\mathbb{R}^d\times U)$,

$$\int_U \int_{\mathbb{R}^d} \zeta(x,u)\,\mathrm{d}\rho_t^{bi}(x,u)\\=\int_{\{y \in \Omega\,\,:\,\,t \leq \tau^{T^{bi}(y),y}\}} \zeta(\gamma^{T^{bi}(y),y}(t),u^{T^{bi}(y),y}(t))\,\mathrm{d}\nu^-(y).$$ 
One has
\begin{equation*}
\int_{\mathbb{R}^+}\int_{\mathbb{R}^d} \psi(x)\,\mathrm{d}\eta^{bi}(t,x)=\int_{\Omega} \psi(y)\,\mathrm{d}\nu^-(y),\,\,\mbox{for every}\,\,\psi \in C_0(\mathbb{R}^d),
\end{equation*}
and $$\int_{\mathbb{R}^+} \int_U \int_{\mathbb{R}^d} \left[\partial_t \xi(t,x) + k(t,x,u) \cdot \nabla \xi(t,x)\right]\,\mathrm{d}\rho_t^{bi}(x,u)\,\mathrm{d}t$$ 
\begin{align*}
&=\int_{\Omega}\int_{0}^{\tau^{T^{bi}(y),y}} \left[\partial_t \xi(t,\gamma^{T^{bi}(y),y}(t)) + k(t,\gamma^{T^{bi}(y),y}(t),u^{T^{bi}(y),y}(t)) \cdot \nabla \xi(t,\gamma^{T^{bi}(y),y}(t))\right]\mathrm{d}t\,\mathrm{d}\nu^-\\ 
&\qquad\qquad=\,\,\,\int_{\mathbb{R}^+}\int_{\mathbb{R}^d} \xi(t,x)\,\mathrm{d}\eta^{bi}(t,x)- \int_{\Omega} \xi(0,T^{bi}(y))\,\mathrm{d}\nu^-(y)\\ 
&\qquad\qquad=\,\,\,\int_{\mathbb{R}^+}\int_{\mathbb{R}^d} \xi(t,x)\,\mathrm{d}\eta^{bi}(t,x)- \int_{\Omega} \xi(0,x)\,\mathrm{d}\left[T^{bi}_{\#}\nu^-\right](x),\,\,\mbox{for every}\,\,\,\xi \in C_0^1(\mathbb{R}^+ \times \mathbb{R}^d).
\end{align*}
Moreover, we have 
\begin{eqnarray*}
&& \int_{\mathbb{R}^+} \int_U \int_{\mathbb{R}^d} L(t,x,u)\,\mathrm{d}\rho_t^{bi}(x,u)\,\mathrm{d}t + \int_{\partial\Omega}\int_U \psi^+(x)\,\mathrm{d}\rho_0^{bi}(x,u)\\ 
&&\qquad =\,\,\int_{\Omega}\int_{0}^{\tau^{T^{bi}(y),y}} L(t,\gamma^{T^{bi}(y),y}(t),u^{T^{bi}(y),y}(t))\,\mathrm{d}t\,\mathrm{d}\nu^-(y) + \int_{\partial\Omega} \psi^+(x)\,\mathrm{d}\left[T^{bi}_{\#}\nu^-\right](x)\\ \\
&&\qquad =\,\,\int_{\Omega}c(T^{bi}(y),y)\,\mathrm{d}\nu^-(y) + \int_{\partial\Omega} \psi^+(x)\,\mathrm{d}\left[T^{bi}_{\#}\nu^-\right](x). \qedhere 
\end{eqnarray*}
\end{proof}
\noindent Now, set
$$\mathcal{E}_0(\mu^+,\mu^-):=\bigg\{(\rho,\eta) \in \mathcal{E}(\mu^+ + \chi^+,\mu^- + \chi^-) \,:\,\chi^+,\,\chi^- \in \mathcal{M}^+(\partial\Omega)\bigg\}.$$
In other words, a pair $(\rho,\eta)$ belongs to $\mathcal{E}_0(\mu^+,\mu^-)$ if and only if we have the following: for all $\psi \in C_0(\Omega)$,
$$\int_{\Omega} \int_{\mathbb{R}^+} \psi(x)\,\mathrm{d}\eta(t,x)
=\int_{\Omega} \psi(y)\,\mathrm{d}\mu^-(y),$$\\
while, for all $\xi \in C_0^1(\mathbb{R}^+ \times\mathbb{R}^d)$ such that $\xi(0,\cdot)=0$ on $\partial\Omega$, one has
$$\int_{\mathbb{R}^+} \int_U \int_{\mathbb{R}^d} \left[\partial_t \xi(t,x) + k(t,x,u) \cdot \nabla \xi(t,x)\right]\mathrm{d}\rho_t(x,u)\,\mathrm{d}t=\int_{\mathbb{R}^+}\int_{\mathbb{R}^d} \xi(t,x)\,\mathrm{d}\eta(t,x)- \int_{\Omega} \xi(0,x)\,\mathrm{d}\mu^+(x).$$
Now, we have the following result.
\begin{proposition} \label{Eulerian and Kantorovich with the same cost}
There exists an admissible pair $(\rho,\eta) \in \mathcal{E}_0(\mu^+,\mu^-)$ such that  
\begin{eqnarray*}
&&\int_{\mathbb{R}^+} \int_U \int_{\mathbb{R}^d} L(t,x,u)\,\mathrm{d}\rho_t(x,u)\,\mathrm{d}t +  \int_{\partial\Omega}\int_U \psi^+(x)\,\mathrm{d}\rho_0(x,u) -  \int_{\partial\Omega}\int_{\mathbb{R}^+} \psi^-(x)\,\mathrm{d}\eta(t,x)\\
&&\qquad \qquad =\int_{\Omega \times \Omega}c(x,y)\,\mathrm{d}\pi(x,y) + \int_{\partial\Omega} \psi^+(x)\,\mathrm{d}\pi_x - \int_{\partial\Omega} \psi^-(y)\,\mathrm{d}\pi_y.
\end{eqnarray*}
Consequently, 
\begin{equation}\label{T2:T1}
{\mathcal T}_2(\mu^+, \mu^-; \psi^+, \psi^-) \leq {\mathcal T}_1(\mu^+, \mu^-; \psi^+, \psi^-).
\end{equation}
\end{proposition}
\begin{proof}
Set 
$\rho:=\rho^{ii} + \rho^{ib} + \rho^{bi} \,\,\,\mbox{and}\,\,\,\eta:=\eta^{ii} + \eta^{ib} + \eta^{bi}.$
From Lemmas \ref{Eulerian ii}, \ref{Eulerian ib} \& \ref{Eulerian bi}, we have for all $\psi \in C_0(\Omega),$
\begin{align*}\int_{\mathbb{R}^+}\int_\Omega \psi(x)\,\mathrm{d}\eta(t,x)
&=\int_{\mathbb{R}^+}\int_\Omega \psi(x)\,\mathrm{d}\eta^{ii}(t,x)+\int_{\mathbb{R}^+}\int_\Omega \psi(x)\,\mathrm{d}\eta^{ib}(t,x) +\int_{\mathbb{R}^+}\int_\Omega \psi(x)\,\mathrm{d}\eta^{bi}(t,x)\\
&=\int_{\Omega} \psi(y)\,\mathrm{d}\mu^-(y),
\end{align*}
while, for all $\xi \in C_0^1(\mathbb{R}^+ \times\mathbb{R}^d)$ such that $\xi(0,\cdot)=0$ on $\partial\Omega$, one has
\begin{eqnarray*} 
&&\int_{\mathbb{R}^+} \int_U \int_{\mathbb{R}^d} \left[\partial_t \xi(t,x) + k(t,x,u) \cdot \nabla \xi(t,x)\right]\mathrm{d}\rho_t(x,u)\,\mathrm{d}t\\
&&\qquad=\int_{\mathbb{R}^+} \int_U \int_{\mathbb{R}^d} \left[\partial_t \xi(t,x) + k(t,x,u) \cdot \nabla \xi(t,x)\right]\mathrm{d}\left[\rho_t^{ii} + \rho_t^{ib} + \rho_t^{bi}\right](x,u)\,\mathrm{d}t\\
&& \qquad=\int_{\mathbb{R}^+}\int_{\mathbb{R}^d} \xi(t,x)\,\mathrm{d}\eta^{ii}(t,x)- \int_{\Omega} \xi(0,x)\,\mathrm{d}(\mu^+ - \nu^+)(x) + \int_{\mathbb{R}^+}\int_{\mathbb{R}^d} \xi(t,x)\,\mathrm{d}\eta^{ib}(t,x)\\
&&\qquad \qquad\quad- \int_{\Omega} \xi(0,x)\,\mathrm{d}\nu^+(x) + \int_{\mathbb{R}^+}\int_{\mathbb{R}^d} \xi(t,x)\,\mathrm{d}\eta^{bi}(t,x)- \int_{\Omega} \xi(0,x)\,\mathrm{d}\left[T^{bi}_{\#}\nu^-\right](x)\\
&& \qquad=\int_{\mathbb{R}^+}\int_{\mathbb{R}^d} \xi(t,x)\,\mathrm{d}\eta(t,x)- \int_{\Omega} \xi(0,x)\,\mathrm{d}\mu^+(x).
\end{eqnarray*}

\noindent On the other hand, by Lemmas \ref{Eulerian ii}, \ref{Eulerian ib} \& \ref{Eulerian bi}, we get
{\small
\begin{eqnarray*}
&&\int_{\mathbb{R}^+} \int_U \int_{\mathbb{R}^d} L(t,x,u)\,\mathrm{d}\rho_t(x,u)\,\mathrm{d}t +  \int_{\partial\Omega}\int_U \psi^+(x)\,\mathrm{d}\rho_0(x,u) - \int_{\partial\Omega} \int_{\mathbb{R}^+} \psi^-(x)\,\mathrm{d}\eta(t,x)\\
&&\quad\qquad=\int_{\Omega \times \Omega}c(x,y)\,\mathrm{d}\pi^{ii}(x,y)+ \int_{\Omega}c(T^{bi}(y),y)\,\mathrm{d}\nu^-(y) + \int_{\partial\Omega} \psi^+(x)\,\mathrm{d}\left[T^{bi}_{\#}\nu^-\right](x)\\
&&\qquad\qquad\qquad \, +\int_{\Omega}c(x,T^{ib}(x))\,\mathrm{d}\nu^+(x) - \int_{\partial\Omega} \psi^-(y)\,\mathrm{d}\left[T^{ib}_{\#}\nu^+\right](y)\\
&& \quad \qquad=\int_{\Omega \times \Omega}c(x,y)\,\mathrm{d}\pi(x,y) + \int_{\partial\Omega} \psi^+(x)\,\mathrm{d}\pi_x - \int_{\partial\Omega} \psi^-(y)\,\mathrm{d}\pi_y.
\end{eqnarray*}
This completes the proof. $\qedhere$
}
\end{proof}
We now recall the Eulerian dual problem
{\small
\begin{equation} \label{Problem 11}
{\mathcal D}_2(\mu^\pm; \psi^\pm):=\sup\left\{\int_\Omega \varphi^-\mathrm{d}\mu^- - \int_\Omega J^+(0,\cdot)\,\mathrm{d}\mu^+: 
\begin{array}{l}\varphi^- \in C(\mathbb{R}^d),\,J^+ \in C^1(\mathbb{R}^+ \times \mathbb{R}^d),\\
(J^+,\varphi^-)\;\,\mbox{solves}\,\;\eqref{Hamilton} \end{array}\right\},
\end{equation}
}
\hspace{-0.5em} where
\begin{equation} \label{Hamilton}
\begin{cases} 
 \partial_t J^+(t,x) + k(t,x,u)\cdot \nabla J^+(t,x) \leq L(t,x,u) \,\,&\,\,\mbox{for all}\,\,\,(t,x,u) \in \mathbb{R}^+ \times \mathbb{R}^d \times U,\\
 \varphi^-(x) \leq J^+(t,x)\,\,&\,\,\mbox{for all}\,\,\,(t,x) \in \mathbb{R}^+ \times \mathbb{R}^d,\\
\psi^-\leq \varphi^- \leq J^+(0,\cdot) \leq \psi^+ \,\,&\,\,\mbox{on}\,\,\,\partial\Omega,
 \end{cases}
 \end{equation}\\
 and prove the following.
\begin{proposition} Under the above assumptions, we have 
\begin{equation}\label{D1:D2}
{\mathcal D}_1(\mu^+, \mu^-; \psi^+, \psi^-) \leq {\mathcal D}_2(\mu^+, \mu^-; \psi^+, \psi^-).
\end{equation}
\end{proposition}
\begin{proof}
Set for all $(t,x) \in \mathbb{R}^+ \times \mathbb{R}^d$, 
$$J^+_{\varphi^-}(t,x)=-\min_{u,\,\tau}\bigg\{\int_t^{\tau}L(s,\gamma^{t,x}_u(s),u(s))\,\mathrm{d}s \,-\, \varphi^-(\gamma^{t,x}_u(\tau)) \bigg\}.$$
It is well known that $J^+_{\varphi^-}$ is a viscosity solution of the following
Hamilton-Jacobi-Bellman quasi-variational inequality (see, for instance, \cite{Bardi}):
\begin{equation} \label{HJBV}
\begin{cases}
 \partial_t J + H(t,x,\nabla J) \leq 0\,\,&\,\mbox{on}\,\,\,\mathbb{R}^+ \times \mathbb{R}^d,\\
 \varphi^- \leq J\,\,&\,\mbox{on}\,\,\,\mathbb{R}^+ \times \mathbb{R}^d,\\
 \partial_t J + H(t,x,\nabla J) = 0\,\,&\,\mbox{on}\,\,\,\{ \varphi^- < J\},
 \end{cases}
 \end{equation}
 where  $$H(t,x,p):=\max_u \{k(t,x,u) \cdot p - L(t,x,u)\},\,\,\mbox{for every}\,\,(t,x,p)\in\mathbb{R}^+ \times \mathbb{R}^d \times \mathbb{R}^d.$$
 Let $(\varphi^+,\varphi^-)$ be an admissible pair of the dual problem ${\mathcal D}_1(\mu^\pm; \psi^\pm)$ in \eqref{dual with boundary costs}. By the definition of $J^+_{\varphi^-}$, it is clear that we have
 $$J^+_{\varphi^-}(0,\cdot) \leq \varphi^+ \hbox{\quad  and \quad $\left[-J^+_{\varphi^-}(0,\cdot)\right] \oplus \varphi^- \leq c.$}$$
On the other hand, it is also known \cite{Bardi} that the viscosity solution $J^+_{\varphi^-}$ is the
infimum of all supersolutions $J$ to
$$\begin{cases}
 \partial_t J(t,x) + k(t,x,u)\cdot \nabla J(t,x) \leq L(t,x,u) \,\,&\,\,\mbox{for all}\,\,\,(t,x,u) \in \mathbb{R}^+ \times \mathbb{R}^d \times U,\\
 \varphi^-(x) \leq J(t,x)\,\,&\,\,\mbox{for all}\,\,\,(t,x) \in \mathbb{R}^+ \times \mathbb{R}^d,\\
 \psi^-(x) \leq \varphi^- (x) \,\,\, &\ x \in \partial\Omega,\\
J^+(0,x) \leq \psi^+(x) \,\,\, &\ x \in \partial\Omega.
 \end{cases}
 $$
 As a consequence, we get
 {\small
 \begin{eqnarray*}
{\mathcal D}_1(\mu^\pm; \psi^\pm) 
 &\leq & \,\,\sup_{\varphi^-}\left\{\int_\Omega \varphi^- \,\mathrm{d} \mu^- - \int_\Omega J_{\varphi^-}^+(0,\cdot) \,\mathrm{d}\mu^+
:\psi^- \leq \varphi^- \leq J_{\varphi^-}^+(0,\cdot)\leq \psi^+\, \mbox{on}\;
\partial\Omega\right\}\\
&\leq &{\mathcal D}_2(\mu^\pm; \psi^\pm). \qedhere
\end{eqnarray*}
}
 \end{proof}
 On the other hand, we have the following:
 \begin{proposition}\label{the Eulerian dual is less than the Eulerian primal}
 The following inequality holds:
\begin{equation}\label{D2:T2}
{\mathcal D}_2(\mu^+, \mu^-; \psi^+, \psi^-) \leq {\mathcal T}_2(\mu^+, \mu^-; \psi^+, \psi^-).
\end{equation}
 \end{proposition}
 \begin{proof}
 For an admissible $(\rho,\eta) \in \mathcal{E}_0(\mu^+,\mu^-)$, there are two nonnegative measures $\chi^+$ and $\chi^-$ which are concentrated on the boundary such that $(\rho,\eta) \in \mathcal{E}(\mu^+ + \chi^+,\mu^- + \chi^-)$. Let  $(J^+,\varphi^-)$ satisfy \eqref{Hamilton} with $\varphi^- \in C(\mathbb{R}^d)$ and $J^+ \in C^1(\mathbb{R}^+ \times \mathbb{R}^d)$. Then, we have 
\begin{align*}
&\qquad\qquad\qquad\qquad\qquad\qquad\int_{\Omega} \varphi^-(y)\,\mathrm{d}\mu^-(y) - \int_{\Omega} J^+(0,x)\,\mathrm{d}\mu^+(x)\\
&= \int_{\mathbb{R}^+}\int_{\mathbb{R}^d} \varphi^-(x)\,\mathrm{d}\eta(t,x) -\int_{\Omega} \varphi^-(y)\,\mathrm{d}\chi^-(y) + \int_{\Omega} J^+(0,x)\,\mathrm{d}\chi^+(x) - \int_{\Omega} J^+(0,x)\,\mathrm{d}\left[\mu^+ + \chi^+\right]\\ 
&\leq  \int_{\mathbb{R}^+}\int_{\mathbb{R}^d} J^+(t,x) \mathrm{d}\eta(t,x) -\int_{\Omega} \varphi^-(y) \mathrm{d}\chi^-(y)  + \int_{\Omega} J^+(0,x)\mathrm{d}\chi^+(x) - \int_{\Omega} J^+(0,x) \mathrm{d}\left[\mu^+ + \chi^+\right]\\ 
&=\int_{\mathbb{R}^+} \int_U \int_{\mathbb{R}^d} \left[\partial_t J^+(t,x) + k(t,x,u) \cdot \nabla J^+(t,x)\right]\mathrm{d}\rho_t(x,u)\mathrm{d}t -\int_{\Omega} \varphi^-\mathrm{d}\chi^- + \int_{\Omega} J^+(0,\cdot)\mathrm{d}\chi^+\\ 
&\leq \int_{\mathbb{R}^+} \int_U \int_{\mathbb{R}^d} L(t,x,u)\,\mathrm{d}\rho_t(x,u)\,\mathrm{d}t  +  \int_{\partial\Omega}\int_U \psi^+(x)\,\mathrm{d}\rho_0(x,u) -  \int_{\partial\Omega}\int_{\mathbb{R}^+} \psi^-(x)\,\mathrm{d}\eta(t,x).
\end{align*}\\
This implies that
\small{
\begin{eqnarray*} 
&&\sup\left\{\int_\Omega \varphi^- \,\mathrm{d} \mu^- - \int_\Omega J^+(0,\cdot) \,\mathrm{d}\mu^+ 
\,:\,\begin{array}{l}\varphi^- \in C(\mathbb{R}^d),\,J^+ \in C^1(\mathbb{R}^+ \times \mathbb{R}^d),\\
\,(J^+,\varphi^-) \,\,\,\mbox{satisfies}\,\,\,\eqref{Hamilton}\end{array}\right\}\\ \\
&&\leq \min_{(\rho,\eta) \in \mathcal{ E}_0(\mu^+,\mu^-)}\bigg\{\int_{\mathbb{R}^+} \int_U \int_{\mathbb{R}^d} L(t,x,u)\,\mathrm{d}\rho \,\,
+\,\, \int_{\partial\Omega}\int_U \psi^+(x)\,\mathrm{d}\rho_0 \,\,
-\,\, \int_{\partial\Omega}\int_{\mathbb{R}^+} \psi^-(x)\,\mathrm{d}\eta \bigg\}. 
\end{eqnarray*}\\
This concludes the proof. $\qedhere$
}
\end{proof}
Finally, we obtain the following.
\begin{theorem} \label{equalities between several problems} Under the above assumptions on $L$, $c$, $\psi^+$ and $\psi^-$,
\begin{enumerate}
\item The following equalities hold:
\begin{equation}\label{all}
{\mathcal T}_1(\mu^\pm; \psi^\pm) ={\mathcal D}_1(\mu^\pm; \psi^\pm) ={\mathcal D}_2(\mu^\pm; \psi^\pm) = {\mathcal T}_2(\mu^\pm; \psi^\pm).
\end{equation} 
\item Let $\pi$ be an optimal plan for ${\mathcal T}_1$ and its associated decomposition $\pi=\pi^{ii} + \pi^{ib} + \pi^{bi}$. Then, the corresponding admissible pairs $(\rho^{ii},\eta^{ii}),\,(\rho^{ib},\eta^{ib})$ and $(\rho^{bi},\eta^{bi})$ are, respectively, solutions for the Eulerian problems:
\begin{equation} 
\min\bigg\{\int_{\mathbb{R}^+} \int_U \int_{\mathbb{R}^d} L(t,x,u)\,\mathrm{d}\rho\,:\,(\rho,\eta) \in\mathcal{ E}(\mu^+ - \nu^+,\mu^- - \nu^-)\bigg\},
\end{equation} 
\begin{equation} \min\bigg\{\int_{\mathbb{R}^+} \int_U \int_{\mathbb{R}^d} L(t,x,u)\,\mathrm{d}\rho \,\,-\,\, \int_{\partial\Omega} \psi^-\,\mathrm{d}\chi^-\,:\,(\rho,\eta) \in\mathcal{ E}(\nu^+,\chi^-),\,\spt(\chi^-) \subset \partial\Omega\bigg\},
\end{equation} 
and
\begin{equation} \min\bigg\{\int_{\mathbb{R}^+} \int_U \int_{\mathbb{R}^d} L(t,x,u)\,\mathrm{d}\rho \,\,+\,\, \int_{\partial\Omega} \psi^+\,\mathrm{d}\chi^+\,:\,(\rho,\eta) \in\mathcal{ E}(\chi^+,\nu^-),\,\spt(\chi^+) \subset \partial\Omega\bigg\}.
\end{equation} 
\item The pair $(\rho,\eta)$, where $\rho=\rho^{ii} + \rho^{ib} + \rho^{bi}$ and $\eta=\eta^{ii} + \eta^{ib} + \eta^{bi}$, solves the Eulerian problem
\begin{equation*}
\min_{(\rho,\eta) \in\mathcal{ E}_0(\mu^+,\mu^-)}\bigg\{\int_{\mathbb{R}^+} \int_U \int_{\mathbb{R}^d} L(t,x,u)\,\mathrm{d}\rho \,\,+\,\, \int_{\partial\Omega}\int_U \psi^+(x)\,\mathrm{d}\rho_0 \,\,-\,\, \int_{\partial\Omega}\int_{\mathbb{R}^+} \psi^-(x)\,\mathrm{d}\eta\bigg\}.
\end{equation*}
\end{enumerate}
\end{theorem}
\begin{proof}
For (\ref{all}) it suffices to combine (\ref{T1:D1}), (\ref{T2:T1}), (\ref{D1:D2}) and (\ref{D2:T2}). The rest follows directly from the optimality of the transport plan $\pi=\pi^{ii} + \pi^{ib} + \pi^{bi}$, the optimalities of $\pi^{ii},\,\pi^{ib}$ and $\pi^{bi}$ in problems \ref{(P1)}, \ref{(P2)} \& \ref{(P3)}, respectively, 
 the equalities in (\ref{all}), as well as 
Lemmas \ref{Eulerian ii}, \ref{Eulerian ib} \& \ref{Eulerian bi} and Proposition \ref{Eulerian and Kantorovich with the same cost}. $\qedhere$
\end{proof}
\section{Uniqueness of optimal transport plans and associated stopping times}\label{Sec.5}

 In this section, we show that, under additional assumptions on the Hamiltonian $H$, the optimal transport plan (resp. stopping plan) (from interior to interior) $\pi^{ii}$ (resp. $\eta^{ii}$) is a transport map (resp. a stopping time). More precisely, we prove that the optimal transport plan $\pi^{ii}$ is concentrated on a graph $y=T^{ii}(x)$ and the optimal stopping distribution $\eta^{ii}$ will be concentrated on a graph $(t,y)=(\tau,T^{ii})(x)$, as soon as the source measure $\mu^+$ is absolutely continuous with respect to the Lebesgue measure $\mathcal{L}^d$.  The map is constructed by solution to the Hamiltonian flow with the end time determined by the Pontryagin transversality condition. The argument will be sketched as this is a straight forward extension of the result of \cite{Aaron}. We also show that the maps $T^{ib}$ and $T^{bi}$ are unique under these assumptions, and given by similar constructions.

Let $(J^+(0,\cdot)\,,\varphi^-)$ be a solution for the dual problem \eqref{Eulerian classique with boundary} between $\mu^+ - \nu^+$ and $\mu^- - \nu^-$. Recall that the potential $J^+$ is given by the following formula (for more details about the following optimal control problem with free stopping time, we refer to \cite{Bardi}): 
$$J^+(t,x)=-\min_{u,\,\tau}\bigg\{\int_t^{\tau}L\big(s,\gamma^{t,x}_u(s),u(s)\big)\,\mathrm{d}s \,-\, \varphi^-(\gamma^{t,x}_u(\tau)) \bigg\},\,\,\,\forall\,\,(t,x) \in \mathbb{R}^+ \times \mathbb{R}^d.$$
Let $\gamma:=\gamma^{x,y}$ be an optimal trajectory from a point $x$ to another one $y$, $u:=u^{x,y}$ be the corresponding optimal control and $\tau:=\tau^{x,y}$ be the corresponding stopping time. First, assume that:\smallskip

(H7)\quad The Hamiltonian $H$ is $C_{loc}^{1,1}(\mathbb{R}^+ \times \mathbb{R}^d \times U)$,

(H8)\quad $u\mapsto k(t,x,u) \cdot p - L(t,x,u)$ has a unique maximizer for all $(t,x,p) \in \mathbb{R}^+ \times \mathbb{R}^d \times \mathbb{R}^d$.\smallskip

\noindent From the Pontryagin maximum principle (see, for instance, \cite{Clarke,Lev}), for any $u,\gamma,\tau$ that minimize the cost
$$
  \int_0^\tau L\big(s,\gamma(s),u(s)\big)ds -\varphi^-\big(\gamma(\tau)\big)+ J^+\big(0,\gamma(0)\big),
$$
there is a Lipschitz continuous arc $p:\left[0,\tau\right] \mapsto\mathbb{R}^d$ that solves: 
$$\dot{p}(t)=-p(t)^{T} \nabla_x k\big(t,\gamma(t),u(t)\big) + \nabla_x L\big(t,\gamma(t),u(t)\big),\,\,\mbox{for a.e.}\,\,\,t \in (0,\tau),$$
and satisfies the maximum principle:
$$H\big(t,\gamma(t),p(t)\big)=p(t) \cdot k(t,\gamma(t),u(t)) - L\big(t,\gamma(t),u(t)\big),$$
with boundary conditions at points of differentiability of $J^+(0,\cdot)$ and $\varphi^-$ (see, for instance, \cite{Cannarsa}):
$$
  p(0)=\nabla J^+\big(0,\gamma(0)\big), \quad p(\tau)=\nabla \varphi^-\big(\gamma(\tau)\big),
$$
and the transversality condition:
$$H\big(\tau,\gamma(\tau),p(\tau)\big)=0.$$
It follows that for a.e.\ $t\in [0,\tau]$
\begin{align*}
\frac{d}{d t}H\big(t,\gamma(t),p(t)\big)=&\ p(t) \cdot \partial_t k\big(t,\gamma(t),u(t)\big) + p(t) \cdot \nabla_x k\big(t,\gamma(t),u(t)\big) \,\dot{\gamma}(t)\\
&\ - \partial_t L\big(t,\gamma(t),u(t)\big) - \nabla_x L\big(t,\gamma(t),u(t)\big)\cdot \dot{\gamma}(t)\\
&\ \qquad+\,\, \dot{p}(t) \cdot k\big(t,\gamma(t),u(t)\big)\\
=&\ p(t) \cdot \partial_t k\big(t,\gamma(t),u(t)\big)  - \partial_t L\big(t,\gamma(t),u(t)\big).
\end{align*}
In addition, we suppose that\smallskip

\qquad \qquad \qquad \qquad\qquad\qquad\qquad(H9) \quad $\partial_t k=0$ on $\mathbb{R}^+ \times \mathbb{R}^d \times U.$\smallskip \\
Then, we get
$$\frac{d}{d t}H\big(t,\gamma(t),p(t)\big) = - \partial_t L\big(t,\gamma(t),u(t)\big).$$
Now, let us assume the following monotonicity condition on the running cost $L$ with respect to time:

\qquad \qquad \qquad \qquad\qquad(H10)  \quad  $\partial_t L > 0$\, or \,$\partial_t L <0 $\, on \,$\mathbb{R}^+ \times \mathbb{R}^d \times U.$ \smallskip\\
Then, this implies that along an optimal trajectory $\gamma$, we have, under the assumption (H10), the following:
\begin{equation} \label{unique stop}
  H\big(t,\gamma(t),p(t)\big)=0 \Leftrightarrow t=\tau.
\end{equation}
On the other hand, we have that, for initial points $\gamma(0)=x$ where $J^+(0,\cdot)$ is differentiable, the pair $(\gamma,p)$ is the unique solution for the following Hamiltonian system:
\begin{equation}\label{Hamiltonian_flow}
\begin{cases}
\dot{\gamma}(t)=\nabla_p H(t,\gamma(t),p(t)),\\
\dot{p}(t)=-\nabla_x H(t,\gamma(t),p(t)),\\
\gamma(0)=x,\\
p(0)=\nabla J^+(0,x).
\end{cases}
\end{equation} 
Hence, the stopping time $\tau$ is uniquely determined by $x$ (almost everywhere): $\tau=\tau(x)$ while the endpoint $y$ of the trajectory $\gamma$ is given by $y=\gamma(\tau(x))$. This constructs the maps $T^+(x)=\gamma(\tau(x))$ (depending on $\nabla J^+(0,x)$ through the initial condition of $p$), where $T^{ii}(x)=T^+(x)$ whenever $T^+(x)\in \Omega$ and $T^{ib}(x)=T^+(x)$ when $T^+(x)\in \partial \Omega$.

In other words, the optimal transport plan $\pi^{ii}$ is concentrated on a graph $y=T^{ii}(x)$ as soon as the source measure $\mu^+$ is absolutely continuous with respect to the Lebesgue measure (for the case $\partial_t L<0$ we also require that $\mu^+$ and $\mu^-$ are disjoint to handle the mass stopping at $\tau=0$). Moreover, the optimal stopping distribution $\eta^{ii}$ is concentrated on the graph $(t,y)=(\tau^{x,T^{ii}(x)},T^{ii}(x))=(\tau(x),T^{ii}(x))$. Then, we have
$$\int_{\Omega \times \Omega} c(x,y)\,\mathrm{d}\pi^{ii}(x,y)=\int_\Omega c(x,T^{ii}(x))\,\mathrm{d}(\mu^+ - \nu^+)(x),$$
and
$$\int_{\mathbb{R}^+}\int_{\mathbb{R}^d} \xi(t,y)\,\mathrm{d}\eta^{ii}(t,y)=\int_{\Omega} \xi(\tau^{x,T^{ii}(x)},T^{ii}(x))\,\mathrm{d}(\mu^+ - \nu^+)(x),\,\,\mbox{for all}\,\,\,\xi \in C_0(\mathbb{R}^+ \times \mathbb{R}^d).$$

\noindent Now, we may assume that
$$\varphi^-(y)=\inf\{J^+_{\varphi^-}(t,y)\,:\,t \geq 0\},\,\,\mbox{for every}\,\,y \in \mathbb{R}^d,$$
because replacing $\varphi^-$ with this infimum can only increase the dual value.
First, let us suppose that \,$\partial_t L >0$\, on \,$\mathbb{R}^+ \times \mathbb{R}^d \times U$. In this case, it is easy to see that $\{J^+_{\varphi^-}(t,y)\}_{t \in \mathbb{R}^+}$ is strictly decreasing. Then, we define the {\it{free boundary}} function $\tau^{[-1]}$ as follows:
$$\tau^{[-1]}(y)=\inf\left\{t \geq 0 \,:\,J^+_{\varphi^-}(t,y)=\varphi^-(y)\right\},\,\,\mbox{for every}\,\,y \in \mathbb{R}^d.$$
Let $\gamma$ be an optimal trajectory with $\gamma(\tau)=y$ and suppose that $\varphi^-$ is differentiable at $y$ (which is the case for a.e. $y$). Then, one can show that the following holds (see \cite{Aaron} for more details):
$$H\big(\tau^{[-1]}(y),y,\nabla \varphi^-(y)\big)=0.$$
But, recalling \eqref{unique stop} and the terminal condition from the Pontryagin maximum principle that $p(\tau)=\nabla \varphi^-(y)$, we infer that $\tau^{[-1]}(y)=\tau.$
In this way, the pair $(\gamma,p)$, where $p$ is the associated dual arc with $\gamma$, turns out to be the unique solution of the following system (with reverse time): 
\begin{equation}\label{Hamiltonian_flow reverse}
\begin{cases}
\dot{\gamma}(t)=\nabla_p H\big(t,\gamma(t),p(t)\big),\\
\dot{p}(t)=-\nabla_x H(t,\gamma(t),p(t)),\\
\gamma\big(\tau^{[-1]}(y)\big)=y,\\
p\big(\tau^{[-1]}(y)\big)=\nabla \varphi^-(y).
\end{cases}
\end{equation} 
We may now define the map $T^-(y)=\gamma(0)$, which is the inverse of $T^{ii}$ when $T^-(y)\in \Omega$, and is the unique map $T^{bi}(y)$ when $T^-(y)\in \partial \Omega$.
In other words,
the optimal transport plan $\pi^{ii}$ is concentrated on a graph $x=[T^{ii}]^{-1}(y)$ and the optimal stopping distribution $\eta^{ii}$ is concentrated on a graph $t=\tau^{[-1]}(y)$, provided that the target measure $\mu^- \ll \mathcal{L}^d$. 
 Similarly, one can prove the same result in the case where $\partial_t L <0$ on $\mathbb{R}^+ \times \mathbb{R}^d \times U$. The only difference is that, now, the {\it{free boundary}} function $\tau^{[-1]}$ becomes
$$\tau^{[-1]}(y)=\sup\{t \geq 0 \,:\,J^+_{\varphi^-}(t,y)=\varphi^-(y)\},\,\,\mbox{for every}\,\,y \in \mathbb{R}^d.$$

In order to disregard any transport of goods from boundary onto boundary consumers, we suppose now that 
\begin{equation}\label{strict}
\psi^-(y) - \psi^+(x) < c(x,y),\,\,\,\mbox{for all}\,\,\,x,\,y \in \partial\Omega.
\end{equation}
Then, we have the following:
\begin{theorem}
Suppose that $\mu^+,\,\mu^- \ll \mathcal{L}^d$ and that (\ref{strict}) holds along with (H0)-(H10) with $\partial_tL>0$. Then, Problem  \eqref{Kantorovich with boundary costs} has a unique optimal transport plan, where $\pi^{ii}+\pi^{ib}$ is supported on the set $\{(x,T^+(x))\}_{x\in \Omega}$ and $\pi^{ii}+\pi^{bi}$ is supported on $\{(T^-(y),y)\}_{y\in \Omega}$.  In the case that $\partial_t L<0$, the result holds where the unique optimal transport plan stops all overlapping mass of $\mu^+\wedge \mu^-$ along the diagonal, and the remainder is supported on the graphs of $T^+$ and $T^-$ as above.
\end{theorem}
\begin{proof}
The proof has essentially been done in the discussion leading up to the theorem.  If $(x,y)$ is in the support of $\pi_{ii}$ then we have that $\varphi^-(y)-J^+(0,x)=c(x,y)$ and since 
$$
  0\leq \int_0^\tau L\big(s,\gamma(s),u(s)\big)-\varphi^-\big(\gamma(\tau)\big)+J^+\big(0,\gamma(0)\big)
$$
for all $(u,\gamma,\tau)$, we find that there is an optimal trajectory with $\gamma(0)=x$ and $\gamma(\tau)=y$.  When $\tau>0$ we have from the Pontryagin maximum principle a dual arc $p$ satisfying $p(0)=\nabla J^+(0,x)$ and $p(\tau)=\nabla \varphi^-(y)$ at points of differentiabiltiy of $J^+(0,\cdot)$ and $\varphi^-$, which occur almost everywhere by Rademacher's theorem.
As discussed above, this allows identifying $(\gamma,p)$ with the unique solution to the Hamiltonian system, either forward in time (\ref{Hamiltonian_flow}) or with reverse time (\ref{Hamiltonian_flow reverse}). In the first case we have realized the support of $\pi^{ii}+\pi^{ib}$ as a graph, and in the second case we realize the support of $\pi^{ii}+\pi^{bi}$ as a graph.

To handle the possibility of $\tau=0$, we note this implies that $J^+(0,x)=\varphi^-(x)$. For the case that $\partial_t L>0$, $J^+(t,x)=\varphi^-(x)$ for all $t$ and thus $H(0,x,\nabla J^+(0,x))\leq 0$.  Since the Hamiltonian can only decrease along the trajectory the only solution is with $\tau(x)=0$.  In the case that $\partial_t L<0$, there is the possibility that $H(0,x,\nabla J^+(0,x))< 0$ and there is $\tau(x)>0$.  In fact, in this case, if there is overlapping mass of $\mu^+$ and $\mu^-$ it most stop at $\tau=0$ and the remainder will stop at $\tau(x)$.  If the overlapping mass does not stop, then there $(y,x)$ in the support of $\pi$ with $y\not=x$. We can acheive the same transport with lower cost by swapping the values of $\pi$ at $(y,x)$ and $(x,T^+(x))$ with $(y,T^+(x))$ and $(x,x)$.  Indeed, with the concatenated trajectory from $y$ to $T^+(x)$ we have
$$
  \int_0^{\tau(y)+\tau(x)}L\big(s,\gamma^y(s),u(s)\big)ds<\int_0^{\tau(x)}L\big(s,\gamma^x(s),u(s)\big)ds+\int_0^{\tau(y)}L\big(s,\gamma^y(s),u(s)\big)ds,
$$
showing that the optimizer must stop all overlapping mass at $\tau=0$.

We have now characterized the support of $\pi^{ii}+\pi^{ib}$ as living on the graph of a function, from which it follows that these measures are unique from the fact that the functional in \eqref{Kantorovich with boundary costs} is linear with convex constraints. Uniqueness of $\pi^{bi}$ follows similarly. $\qedhere$
\end{proof}
\section{The general import/export case.} \label{Sec.6}
In this section,
	we briefly consider the variation of the problem where the reserve mass is taken from a set $K^+$ with cost $\psi^+$ and deposited in the set $K^-$ with cost $-\psi^-$, where $K^+$ and $K^-$ are two compact sets of $\mathbb{R}^d$. This variation is directly equivalent to the problem we have studied for $K^+=K^-=\partial \Omega$.
	We define the admissible set of transport plans for this variant to be
	$$
		\mathcal{P}_K(\mu^+,\mu^-)=\bigg\{\pi \in \mathcal{M}^+(\mathbb{R}^d\times \mathbb{R}^d): (\pi_x)_{|\mathbb{R}^d\backslash K^+}=\mu^+,\ (\pi_y)_{|\mathbb{R}^d\backslash K^-}=\mu^-\bigg\}.
	$$
	Then, we consider the following variant of \eqref{Problem 1}
	$$
		\min\bigg\{\int_{\mathbb{R}^d\times \mathbb{R}^d} c(x,y)d\pi+\int_{K^+}\psi^+d\pi_x-\int_{K^-}\psi^- d\pi_y : \pi\in \mathcal{P}_K(\mu^+,\mu^-)\bigg\}.
	$$
	Again, we assume that the costs $\psi^+$ and $\psi^-$ satisfy the no arbitrage assumption (\ref{g_1g_2}), which becomes
	$$\psi^-(y)-\psi^+(x)\leq c(x,y),\ \hbox{for all $(x,y)\in K^+ \times K^-$}.
	$$
As in Proposition \ref{existence of a minimizer}, we get, under this assumption, the following duality result:
{\small 
\begin{align*} 
& \min \left\{\int_{\Omega\times\Omega}c(x,y)\,\mathrm{d}\pi +
\int_{K^+}\psi^+\,\mathrm{d}\pi_x - \int_{K^-}\psi^-\,\mathrm{d}\pi_y\,:\,\pi \in \mathcal{P}_K (\mu^+,\mu^-)\right\}\\
\qquad=&\ \sup_{\varphi^\pm \in C(\Omega)}\left\{\int_\Omega \varphi^- \mathrm{d} \mu^- - \int_\Omega \varphi^+ \mathrm{d}\mu^+
: \begin{array}{l}\psi^- \leq \varphi^-\,\,\mbox{on}\,\,\,
K^-,\,\,\varphi^+ \leq \psi^+\,\,\mbox{on}\,\,\,
K^+,\\
(-\varphi^+) \oplus \varphi^- \leq c\end{array}\right\}.
\end{align*} 
}
On the other hand, one can also find an Eulerian formulation 
which becomes:
$$\min_{(\rho,\eta) \in \mathcal{E}_K(\mu^+,\mu^-)}\bigg\{\int_{\mathbb{R}^+} \int_U \int_{\mathbb{R}^d} L(t,x,u)\,\mathrm{d}\rho \,\,+\,\, \int_{K^+}\int_U \psi^+(x)\,\mathrm{d}\rho_0 \,\,-\,\, \int_{K^-} \int_{\mathbb{R}^+}\psi^-(x)\,\mathrm{d}\eta\bigg\},$$
where 
$$ \mathcal{E}_K(\mu^+,\mu^-):=\bigg\{(\rho,\eta) \in \mathcal{E}(\mu^+ + \chi^+,\mu^- + \chi^-) \,:\,\chi^+ \in \mathcal{M}^+(K^+) \,\,\,\,\mbox{and}\,\,\,\chi^- \in \mathcal{M}^+(K^-)\bigg\}.$$
Finally, the dual of the Eulerian problem can be expressed as:
	$$
		\sup_{\varphi^-\in C(\mathbb{R}^d)}\bigg\{\int_{\mathbb{R}^d} \varphi^- d\mu^- -\int_{\mathbb{R}^d}J^+(0,\cdot)d\mu^+\ :\ \begin{array}{l} \psi^- \leq \varphi^-\,\,\mbox{on}\,\,\,
K^-\,\mbox{and}\,\,\,J^+(0,\cdot) \leq \psi^+\,\,\mbox{on}\,\,\,
K^+,\\
(J^+,\varphi^-)\,\,\,\mbox{solves}\,\,\,\eqref{HJBV}
\end{array}\bigg\}.
	$$

\section{An example}\label{Sec.7}
  We consider a simple one dimensional example where the particles move at unit speed and the cost depends only on the end time.  We let $\Omega =(0,2)$, $k(t,x,\pm 1)=\pm 1$, and $L(t,x,\pm 1) = g'(t)$, where $g(0)=0$ and $g$ is increasing. Let us consider the case where $\mu^+$ is the uniform distribution on $[0,1]$ to $\mu^-$ is uniform on $[1,2]$.  We suppose that 
$$
  \psi^-(y) = \begin{cases} -p^- & y=0,\\
  -\infty &  y=2,
  \end{cases}\ \ \ \ \ \psi^+(x) = \begin{cases} +\infty & x=0,\\
  p^+ &  y=2,
  \end{cases}
$$
(where $+\infty$ could be a sufficiently large finite number).
To satisfy the no-arbitrage condition, 
 we require that
$$
  p^-+p^+\geq -g(2).
$$
The Hamiltonian is 
$$
  H(t,x,p) = |p|-g'(t),
$$
and the unique solution is easily found in the cases when either $g$ is strictly convex or $g$ is strictly concave.\\

\noindent \textbf{Strictly increasing Lagrangian.} Here, we assume that $g$ is strictly convex so that $t\mapsto L(t,x,u)$ is strictly increasing. A simple ansatz for an optimal map is
$$
  \begin{cases}
  T^{ib}(x) = 0, \,\,&\ 0\leq x<x_1;\\
  T^{ii}(x) = x+1-x_1, \,\,&\ x_1\leq x\leq 1;\\
  T^{bi}(y)=2,\,\,&\ 2-x_1< y\leq 2.
  \end{cases}
$$
We have made the decomposition so that $\nu^+$ has density 1 on $[0,x_1)$ and $\nu^-$ has density $1$ on $(2-x_1,2]$. The free boundary (end time) on $[1,2]$ is
$$
  \tau^{[-1]}(y)=\begin{cases}
    1-x_1, \,\,\,\,& 1 \leq y\leq 2-x_1;\\
    2-y, \,\,&\ 2-x_1\leq y\leq 2.
    \end{cases}
$$
Now, we can solve $\varphi^-$ from the equations,\\
$$
  \begin{cases}
  0 = |\nabla \varphi^-(y)|-g'\big(\tau^{[-1]}(y)\big), \,\,\,\, 1\leq y\leq 2;\\
  -p^- -g(x_1)=\varphi^-(1)-g(1-x_1), 
  \end{cases}
$$
to get
$$
  \varphi^-(y)=\begin{cases}
    -p^--g(x_1)+g(1-x_1)+ g'(1-x_1)(y-1),&\ 1\leq y\leq 2-x_1;\\
    -p^--2g(x_1)+g(1-x_1)+g'(1-x_1)(1-x_1) + g(2-y), &\ 2-x_1\leq y\leq 2.
  \end{cases}
$$
\\
The total cost is
\begin{align*}
  x_1(p^-+p^+)+2\int_0^{x_1} g(x)dx + (1-x_1) g(1-x_1),
\end{align*}
which we differentiate with respect to $x_1$ to obtain the following optimality criteria 
$$
  \varphi^-(2)=p^+,
$$
which concurs with $\psi^+(2)$ because the point $2$ happens to be in the support of $\mu^-$.

It is now straightforward to calculate $J^+(t,x)$ and verify that this is indeed the optimal solution.
In Figure \ref{fig-1} we illustrate the primal and dual solutions for the case that $g(t)=\frac{1}{2}t^2$, $p^-=p^+=\frac{1}{16}$, in which case $x_1=\frac{1}{2}$. Note that the continuation of $\varphi^-$ outside of the support of $\mu^-$ is arbitrary so long as it is sufficiently small.
\begin{figure}[ht]
    \centering
    \begin{subfigure}[b]{0.5\textwidth}
    \centering
          \includegraphics[width=\linewidth]{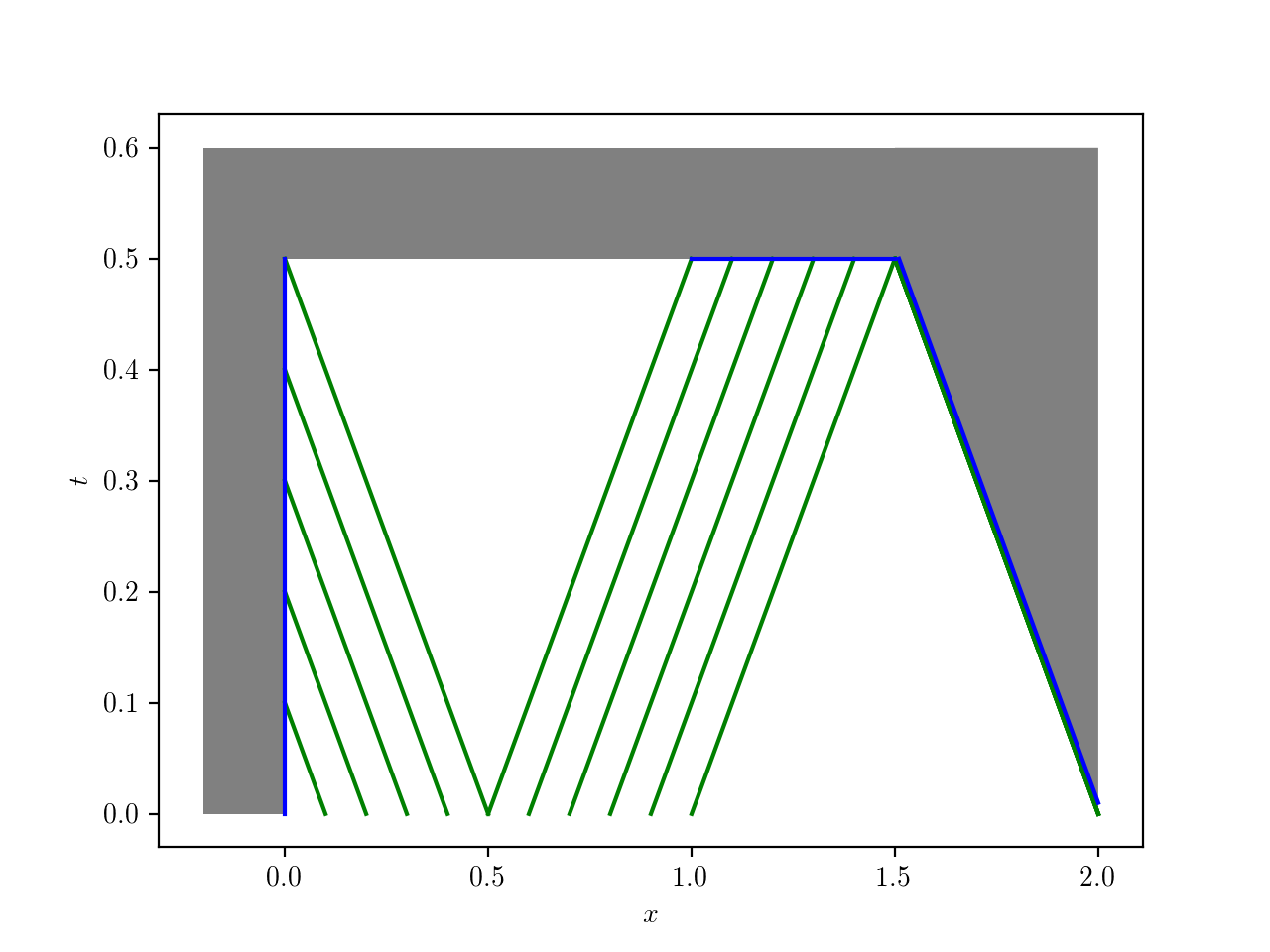}
      \end{subfigure}%
      \begin{subfigure}[b]{0.5\textwidth}
    \centering
          \includegraphics[width=\linewidth]{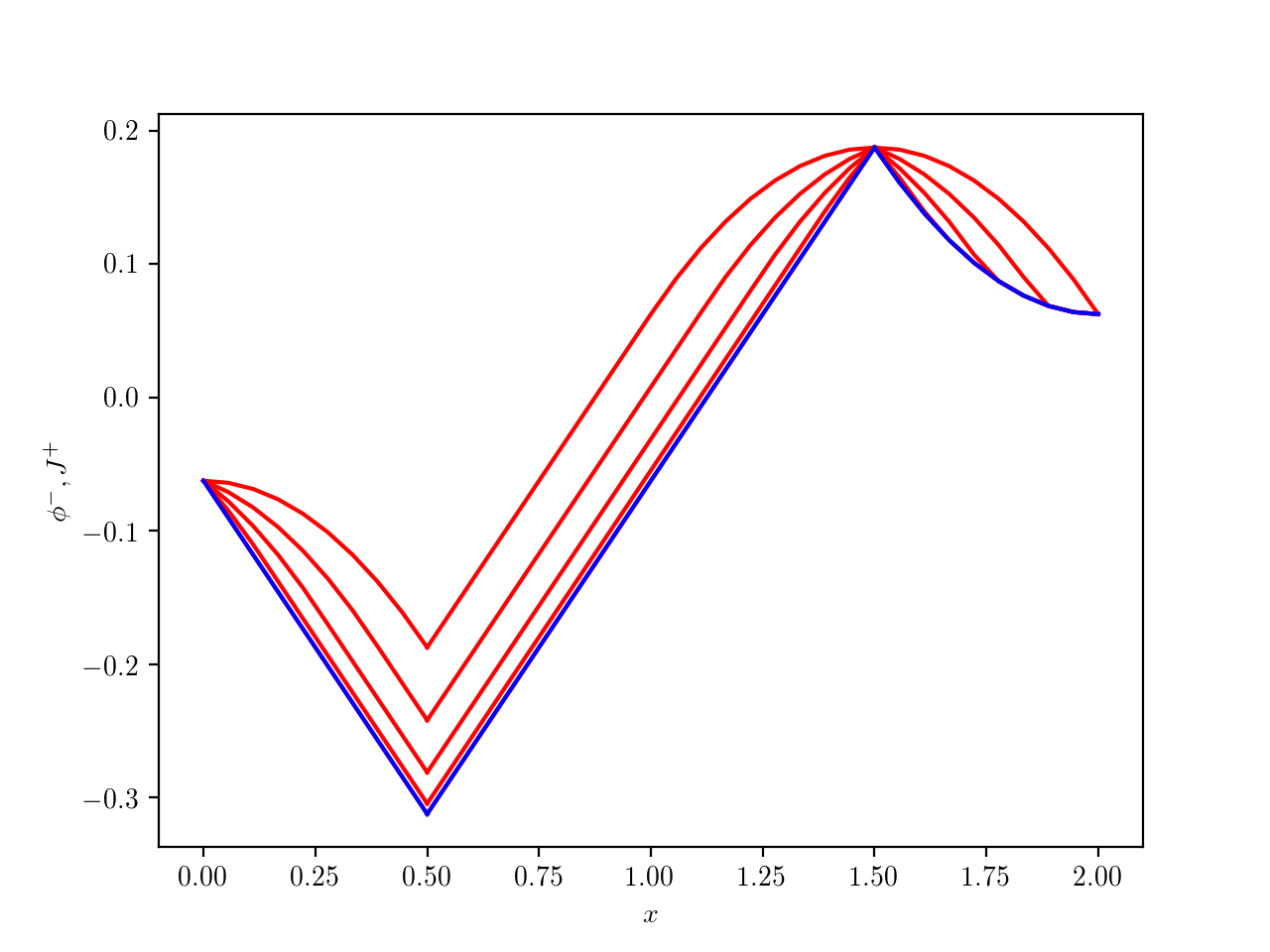}
      \end{subfigure}
      \caption{\label{fig-1} The optimal trajectories are green on the left, which stop upon hitting the free boundary $\tau^{[-1]}$ in blue.  The shaded region shows where $\varphi^-(x)=J^+(t,x)$. On the right $\varphi^-$ is in blue and $J^+(t,\cdot)$ is in red, decreasing as $t\in \{0,0.1,0.2,0.3,0.4\}$. }
   \end{figure}

\noindent \textbf{Strictly decreasing Lagrangian.} We now assume that $g$ is strictly concave so that $t\mapsto L(t,x,u)$ is strictly decreasing. The ansatz for an optimal map is similar, except the interior map reverses orientation,
$$
  \begin{cases}
  T^{ib}(x) = 0, &\ 0\leq x<x_1;\\
  T^{ii}(x) = 2-x, &\ x_1\leq x\leq 1;\\
  T^{bi}(y)=2,&\ 2-x_1< y\leq 2.
  \end{cases}
$$
We find the free boundary to be
$$
  \tau^{[-1]}(y)=\begin{cases}
    2y-2 \,\,& 1 \leq y\leq 2-x_1;\\
    2-y \,\,&\ 2-x_1\leq y\leq 2;
    \end{cases}
$$
and $\varphi^-$ is solved on $[1,2]$ to be
$$
  \varphi^-(y) = \begin{cases} -p^--g(x_1)+g(2-2x_1) +0.5 \,g(2y-2)-0.5 \,g(2-2x_1), &\ 1\leq y\leq 2-x_1;\\
  -p^--2g(x_1)+g(2-2x_1) +g(2-y), &\ 2-x_1< y\leq 2.
  \end{cases}
$$
Similarly, we find optimality occurs when
$$
  \varphi^-(2)=p^+.
$$
In Figure \ref{fig-2}, we plot the solutions for $g(t)=1-e^{-t}$, and $p^-=0$, $p^+\approx -0.15$ (so that $x_1=0.5$).
\begin{figure}[ht]
    \centering
    \begin{subfigure}[b]{0.5\textwidth}
    \centering
          \includegraphics[width=\linewidth]{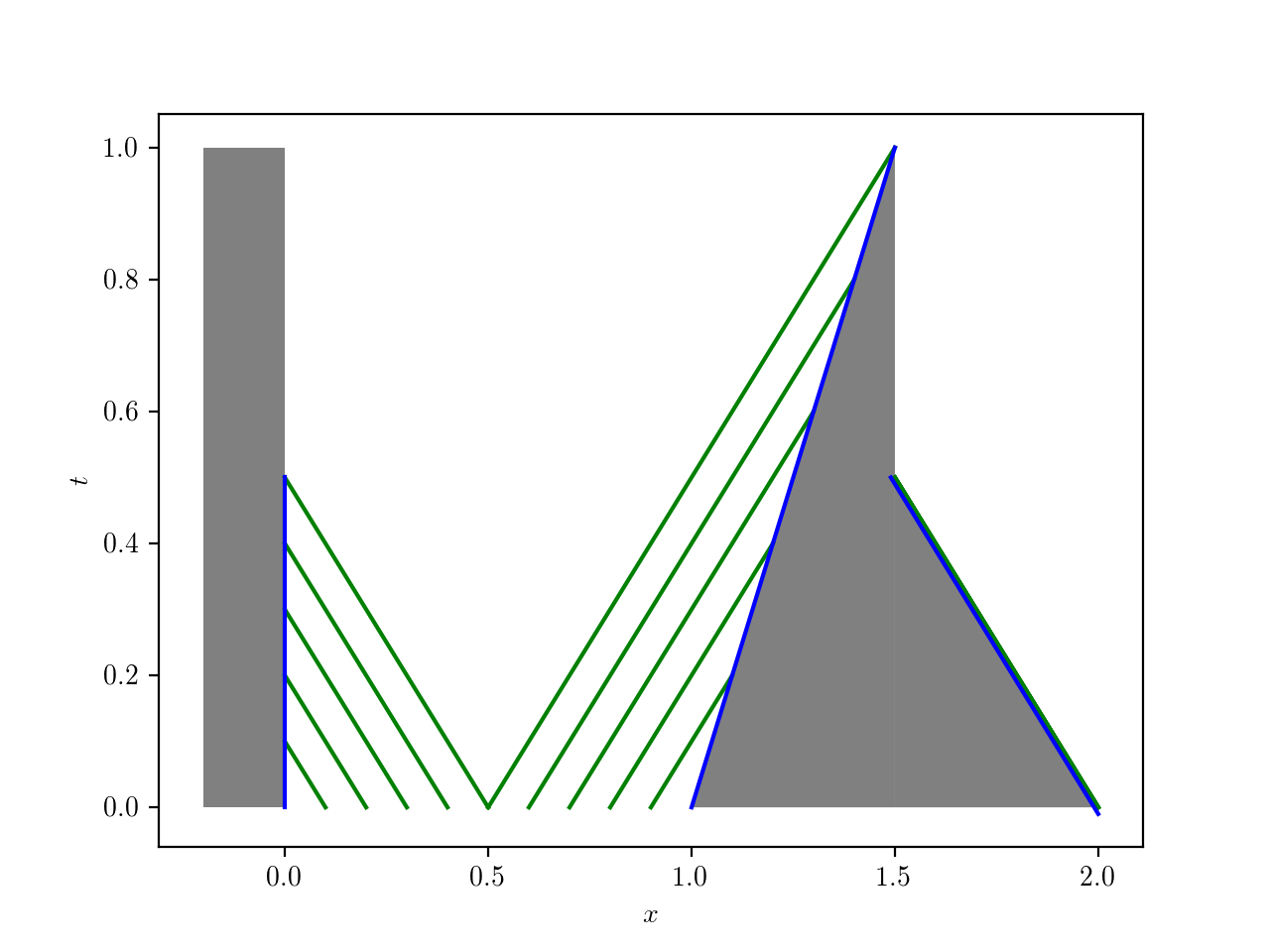}
      \end{subfigure}%
      \begin{subfigure}[b]{0.5\textwidth}
    \centering
          \includegraphics[width=\linewidth]{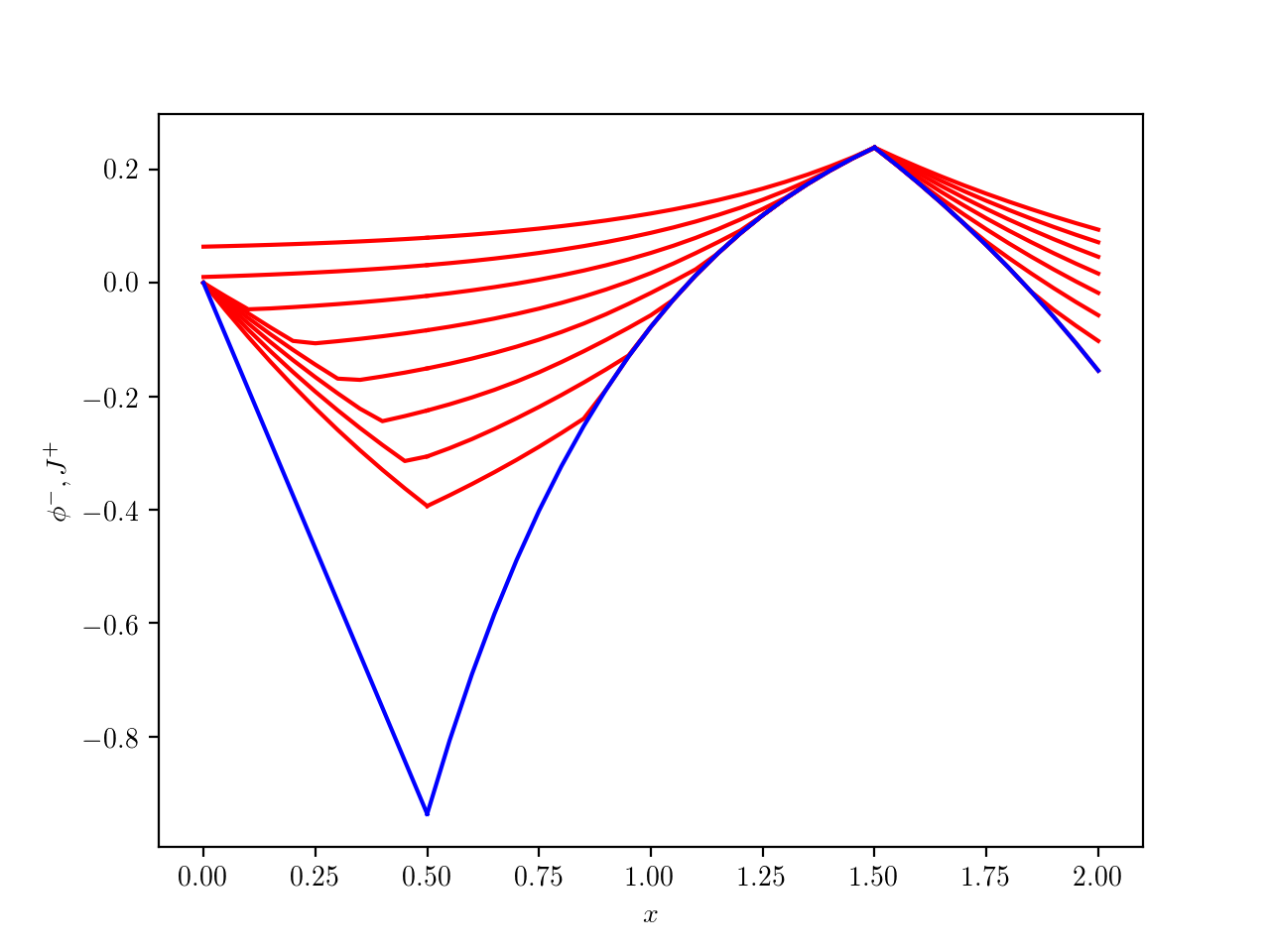}
      \end{subfigure}
      \caption{\label{fig-2} Note that in this case the green optimal trajectories hit the blue free boundary (which is discontinuous) from above. The shaded region again shows where $\varphi^-(x)=J^+(t,x)$. On the right $\varphi^-$ is in blue and $J^+(t,\cdot)$ is in red, increasing for $t$ ranging from $0$ to $\frac{7}{8}$. }
   \end{figure}

\end{document}